\documentclass[11 pt, a4paper, reqno, twoside]{amsart}

\usepackage{amsmath, amsfonts, amssymb, amscd, enumerate, scalefnt}

\usepackage{etex}
\usepackage{t1enc}
\usepackage[latin1]{inputenc}
\usepackage{graphicx}
\usepackage[all]{xy}
\usepackage{color}
\usepackage{slashed}
\usepackage[mathscr]{euscript}
\usepackage{enumitem}
\setlist[itemize]{noitemsep, topsep=1 pt}
\newcommand\bcdot{\ensuremath{
  \mathchoice
   {\mskip\thinmuskip\lower0.2ex\hbox{\scalebox{1.6}{$\cdot$}}\mskip\thinmuskip}}
   {\mskip\thinmuskip\lower0.2ex\hbox{\scalebox{1.6}{$\cdot$}}\mskip\thinmuskip}
   {\lower0.3ex\hbox{\scalebox{1.2}{$\cdot$}}}
   {\lower0.3ex\hbox{\scalebox{1.2}{$\cdot$}}}
}
\theoremstyle{plain}
\newtheorem{theo}{Theorem}[section]

\newtheorem{prop}[theo]{Proposition}

\newtheorem{cor}[theo]{Corollary}
\theoremstyle{definition}
\newtheorem{rem}[theo]{Remark}

\newtheorem{definition}[theo]{Definition}

\theoremstyle{plain}
\newtheorem{lemma}[theo]{Lemma}
\newtheorem{theorem}[theo]{Theorem}

\theoremstyle{definition}

\renewcommand{\=}{:=}

\renewcommand{\a}{\alpha}
\renewcommand{\b}{\beta}
\renewcommand{\c}{\chi}
\renewcommand{\d}{\delta}
\newcommand{\e}{\epsilon}
\newcommand{\f}{\varphi}

\newcommand{\h}{\eta}

\renewcommand{\k}{\kappa}
\renewcommand{\l}{\lambda}
\newcommand{\w}{\omega}
\newcommand{\q}{\vartheta}

\newcommand{\s}{\sigma}
\newcommand{\vars}{\varsigma}
\renewcommand{\t}{\tau}

\newcommand{\x}{\xi}

\newcommand{\D}{\Delta}

\renewcommand{\L}{\Lambda}

\newcommand{\W}{\Omega}

\newcommand{\bC}{\mathbb{C}}

\newcommand{\bR}{\mathbb{R}}
\newcommand{\bZ}{\mathbb{Z}}

\renewcommand{\gg}{\mathfrak{g}}

\newcommand{\gl}{\mathfrak{l}}

\newcommand{\gu}{\mathfrak{u}}

\newcommand{\gX}{\mathfrak{X}}
\newcommand{\gS}{\mathfrak{S}}

\newcommand{\gV}{\mathfrak{V}}

\newcommand{\so}{\mathfrak{so}}

\newcommand\SO{\mathrm{SO}}
\newcommand\SU{\mathrm{SU}}
\newcommand\U{\mathrm{U}}
\newcommand\Spin{\mathrm{Spin}}

\newcommand{\cC}{\mathcal{C}}

\newcommand{\cM}{\mathcal{M}}

\newcommand{\cQ}{\mathcal{Q}}
\newcommand{\cR}{\mathcal{R}}

\newcommand{\cT}{\mathcal{T}}
\newcommand{\cU}{\mathcal{U}}

\newcommand{\eD}{\EuScript{D}}

\newcommand{\eL}{\EuScript{L}}

\newcommand{\eP}{\EuScript{P}}

\newcommand{\eS}{\EuScript{S}}

\newcommand{\eV}{\EuScript{V}}

\newcommand{\p}{\partial}

\newcommand{\ra}{\rightarrow}

\newcommand\Derc[1]{D^{{}_{{}^{#1}}}}
\newcommand\Dirac[1]{\slh{D}^{{}_{{}^{#1}}}_{\color{white} Y}\!}
\newcommand\Diracde{\slh{\partial}_{\color{white} Y}\!\!}

\newcommand{\spin}{\operatorname{spin}}

\renewcommand{\square}{\kern1pt\vbox
{\hrule height 0.6pt\hbox{\vrule width 0.6pt\hskip 3pt \vbox{\vskip
6pt}\hskip 3pt\vrule width 0.6pt}\hrule height0.6pt}\kern1pt}
\renewcommand{\=}{\  \raisebox{0.15mm}{:} {=} \ }

\DeclareMathOperator\Tr{Tr}

\DeclareMathOperator\End{End}

\DeclareMathOperator\Id{Id}
\DeclareMathOperator\Cl{\mathcal{C}\ell}

\newcommand{\Cinf}{\mathcal{C}^{\infty}}

\DeclareMathOperator\Div{\operatorname{Div}}

\renewcommand\Im{\operatorname{Im}}

\newcommand\slh[1]{\slashed #1}

\newcommand{\zero}{\operatorname{o}}
\def\<#1,#2>{\langle\,#1,\,#2\,\rangle}

\newcommand{\Aac}{\`A}

\newcommand{\Math}{{\it Mathematica\raise5 pt\hbox{$\scriptscriptstyle \circledR$}7}}
\newcommand{\beq}{\begin{equation}}
\newcommand{\eeq}{\end{equation}}
\def\<#1,#2>{\langle\,#1,\,#2\,\rangle}
\newcommand{\arr}{\begin{array}{rlll}}
\newcommand{\ea}{\end{array}}
\newcommand{\bea}{\begin{eqnarray}}
\newcommand{\eea}{\end{eqnarray}}
\newcommand{\bean}{\begin{eqnarray*}}
\newcommand{\eean}{\end{eqnarray*}}

\def\sideremark#1{\ifvmode\leavevmode\fi\vadjust{
\vbox to0pt{\hbox to 0pt{\hskip\hsize\hskip1em
\vbox{\hsize3cm\tiny\raggedright\pretolerance10000
\noindent #1\hfill}\hss}\vbox to8pt{\vfil}\vss}}}

\newcounter{ssig}
\newcounter{ttig}
\title[Chern-Dirac bundles on non-K\"ahler Hermitian manifolds]{Chern-Dirac bundles on non-K\"ahler \\ Hermitian manifolds}

\author{Francesco Pediconi}


\begin{document}
\begin{abstract}
We introduce the notions of Chern-Dirac bundles and Chern-Dirac operators on Hermitian manifolds. They are analogues of classical Dirac bundles and Dirac operators, with Levi-Civita connection replaced by Chern connection. We then show that the tensor product of canonical and the anticanonical spinor bundles, called $\eV$-spinor bundle, is a bigraded Chern-Dirac bundle with spaces of harmonic sections isomorphic to the full Dolbeault cohomology class. A similar construction establishes isomorphisms between other types of harmonic sections of the $\eV$-spinor bundle and twisted cohomology.
\end{abstract}

\maketitle


\section{Introduction} \setcounter{equation} 0

A {\it Dirac bundle} over a Riemannian manifold $(M,g)$ is a real or complex vector bundle $\pi: E \ra M$ endowed with a Riemannian or Hermitian metric $h$, a metric connection $D$ and a Clifford multiplication $c: \Cl M \ra \End(E)$, i.e. a structure of left $\Cl M$-module with respect to which the multiplication by tangent vectors is fiber-wise skew-adjoint and covariantly constant. For every such a bundle, there is a distinguished operator, called Dirac operator, which plays a central role in many areas of Differential Geometry and Theoretical Physics (see e.g. \cite{Fr, LM} for an introduction to this topic). The most notable examples of Dirac bundles are spinor bundles on the so called $\spin$ or $\spin^{\bC}$ manifolds.  \smallskip

One of the most important property of Dirac operators is the fact that they are first order, elliptic and formally self-adjoint operators, whose squares have the same principal symbol of the rough Laplacian. On the base of such properties, one may expect the existence of Hodge type theorems for Dirac operators, relating the null spaces of these operators with appropriate cohomology groups of the manifold. This expectation is however contradicted by Hitchin's results in \cite{Hi}, where it is shown that the dimensions of the null spaces of Dirac operators cannot be expressed in purely topological terms. Still, in the special case of K\"ahler geometry, there exist such strong interactions between Clifford multiplications and complex structures which give rise to some notable isomorphisms between the null space of the Dirac operators and certain cohomology groups of the manifold. More precisely, given a compact K\"ahler $2n$-manifold $(M,g,J)$, the following facts hold. \begin{itemize}
\item[i)] $M$ admits a {\it canonical} $\spin^{\bC}$ spinor bundle which is isomorphic to $\L^{0,\bcdot}(T^*M)$ and whose Dirac operator coincides with $\sqrt{2}\big(\bar{\p}+\bar{\p}^*\big)$. From this one gets that the space of harmonic spinors is isomorphic to the Dolbeault cohomology class $H^{0, \bcdot}_{\bar{\p}}(M)$ (see e.g. \cite{Fr}, \S 3.4, or \cite{LM}, Appendix D).
\item[ii)] The complex Clifford bundle $\Cl^{\,\bC}\!M \= \Cl M {\otimes}_{\bR} \bC$ of $M$ always carries a very rich algebraic structure, which has been systematically studied by Michelsohn in \cite{Mi}. There, the author determined a natural bigradation on $\Cl^{\,\bC}\!M$ and, using Dirac operators, constructed a natural elliptic cochain complex which defines the so called {\it Clifford cohomology} of the K\"ahler manifold.
\end{itemize} \smallskip

The aim of this paper is to determine analogues of (i) and (ii) in the more general setting of Hermitian geometry. This is indeed achieved by making use of the Chern connection, instead of the Levi-Civita connection. Following this idea, we first define the {\it Chern-Dirac bundles} on an Hermitian (possibly non-K\"ahler) manifold and the associated {\it Chern-Dirac operators}. They carry all the nice properties of the usual Dirac bundle and Dirac operators, respectively, and they are equal to them in case $M$ is K\"ahler. Then, on any Hermitian manifold $(M,g,J)$, we explicitly construct a distinguished Chern-Dirac bundle $\eV M$, naturally isomorphic to $\Cl^{\,\bC}\!M$, called {\it $\eV$-spinor bundle}. Then, in the same spirit of \cite{Mi}, we show that $\eV M$ is naturally bigraded and that the kernels of the Chern-Dirac operators on $\eV$-spinors are naturally isomorphic to the De Rham and Dolbeault cohomology classes of $M$. Finally, using these new tools, we obtain a spinorial characterization for the Bott-Chern, Aeppli and twisted cohomologies. We also determine explicit expressions for the squares of Chern-Dirac operators, which might be used to determine useful Bochner-type theorems on non-K\"ahler Hermitian manifolds, with appropriate conditions on curvature and torsion. Although we proceed in a similar way, our construction is very different from that given in \cite{Mi}. For the sake of clarity, we will often point out differences and similarities with Michelsohn's framework. \smallskip

The paper is structured as follows. After the first two sections, where some basic properties of Hermitian manifolds, spin groups and $\spin^{\bC}$ structures are recalled, in \S 4 we define Chern-Dirac bundles and Chern-Dirac operators and prove their main properties. In \S 5 we introduce the Chern-Dirac bundle of $\eV$-spinors and prove the main results of this paper. In the final section \S 6 applications of $\eV$-spinors in twisted cohomology are given. \medskip

\noindent{\it Acknowledgement.} We are grateful to Andrea Spiro for useful discussions on several aspects of this paper and his constant support. We also thank Ilka Agricola for helpful comments and suggestions. \bigskip

\section{Preliminaries and notation} \setcounter{equation} 0

In this section we briefly summarize some basic notations and properties of $\spin^{\bC}$ structures and spinors over Hermitian manifolds. We refer to \cite{Fr}, \S 3.4, for a more detailed treatment of these tools. However, we stress the fact that we are using the definition of Clifford algebra of \cite{LM}, based on formula \eqref{relClif}. The sign convention used in \cite{Fr} is opposite to ours and this causes differences in some formulas of this paper from those of that book. 

\subsection{Hermitian manifolds and Chern connections}
Let $(M,g,J)$ be a $2n$-dimensional Hermitian manifold, with fundamental form $\w \!\= g(J\cdot,\cdot)$. The $J$-ho\-lo\-morphic and $J$-antiholomorphic subbundles of $T^{\bC}M$ are denoted by $T^{10}M$ and $T^{01}M$, respectively. Analogously, the corresponding dual subbundles of $T^{*\bC}M$, determined by the $J$-action on covectors $(J\l)(\cdot) \= -\l(J \cdot)$, are denoted by $T^{*10}M$ and $T^{*01}M$, respectively. The bundle of $(p,q)$-forms is indicated by $\L^{p,q}(T^*M)$ and the space of its global sections by $\W^{p,q}(M)$. The decomposition $d= \partial+\bar{\partial}$ is the usual expression of the exterior differential $d$ as sum of the classical $\p$ and $\bar{\p}$ operators. \smallskip 

We denote by $\pi: \SO_g(M) \rightarrow M$ the $\SO_{2n}$-bundle of oriented $g$-orthonormal frames and by $\U_{\!{}_{g,J}\!}(M) \subset \SO_g(M)$ the $\U_n$-subbundle of $(g,J)$-unitary frames, that is of the frames $(e_j)$ satisfying the conditions $g(e_j,e_{\ell})=\d_{j\ell}$ and $Je_{2j-1}=e_{2j}$. Further, for each unitary frame $(e_j) \subset T_xM$, $x \in M$, we denote by \beq \Big(\e_s \= \frac{e_{2s-1}-ie_{2s}}{\sqrt2}, \, \overline{\e}_s \= \frac{e_{2s-1}+ie_{2s}}{\sqrt2}\Big)_{1 \leq s \leq n} \ , \label{vettu} \eeq the complex frame given by the normalized holomorphic and anti-holomorphic parts of the vectors $e_s$. We call it {\it associated normalized complex frame}. \smallskip

Note that, given an unitary frame $(e_j)$ for $T_xM$, the K\"ahler form $\w_x= \w_{{}_{g,J}}|_x$ is equal to $$\w_x = e^1 \!\wedge\! e^2 + {\dots} + e^{2n-1} \!\wedge\! e^{2n} =i\big(\e^1 \!\wedge\! \overline{\e}^1 + {\dots} + \e^n \!\wedge\! \overline{\e}^n\big) \ .$$

The Levi-Civita connection of $(M,g)$ is the torsion-free $\so_{2n}$-valued $1$-form $\w^{{}^{{}_{LC}}}$ on $\SO_g(M)$ and its corresponding covariant derivative on vector fields of $M$ is denoted by $\Derc{LC}_{\color{white} X}$. Similarly, the Chern connection of $(M,g,J)$ is the $\gu_n$-valued connection form $\w^{{}^{{}_{\cC}}}$ on $\U_{\!{}_{g,J}\!}(M)$, whose associated covariant derivative $\Derc{\cC}_{\color{white} X}$\!\! on vector fields of $M$ has torsion satisfying $T(J\cdot,\cdot)=T(\cdot,J\cdot)$. The covariant derivatives of the Levi-Civita and Chern connections are related by \beq \Derc{\cC}_XY=\Derc{LC}_XY+S(X,Y) \,\, , \label{contS} \eeq where $S$ is the uniquely determined {\it contorsion tensor of the Chern connection}. It is well known that the torsion and the contorsion of Chern connection are given by \beq \begin{gathered}
S(X,Y,Z)=-\frac12d\w(JX,Y,Z) \,\, ,\\
T(X,Y,Z)=-\frac12\big(d\w(JX,Y,Z)+d\w(X,JY,Z)\big) \,\, .
\end{gathered} \label{TS2} \eeq and they are related by \beq \begin{gathered}
T(X,Y)=S(X,Y)-S(Y,X) \,\, , \\
2S(X,Y,Z) = T(X,Y,Z)-T(Y,Z,X)+T(Z,X,Y) \,\, ,
\end{gathered} \label{TS} \eeq where $S(X,Y,Z) \= g(S(X,Y),Z)$ and $T(X,Y,Z) \= g(T(X,Y),Z)$. \smallskip

Finally, we recall that the {\it Lee form of $(M,g,J)$} is the 1-form $$\q(X) \= \Tr\!\big(T(X,\cdot)\big) = \sum_j T(X, e_j,e_j) \, ,$$ where $(e_j)$ is an arbitrary choice of a (local) unitary frame field on $M$. One can directly check that the fundamental form $\w$ and the Lee form $\q$ are related by $$\q=-Jd^*\w \, ,$$ where $d^*$ is the adjoint of $d$ with respect to $g$.

\subsection{Complex spin representations and $\Spin_n^{\, \bC}$}
We recall that the Clifford algebra $\Cl_n$ is the real associative algebra with unit generated by $n$ elements $(e_j)$ satisfying \beq e_j \!\cdot\! e_k + e_k \!\cdot\! e_j=-2\d_{jk} \,\, , \quad \,\,\,\,1 \leq j,k \leq n \, . \label{relClif} \eeq As vector space, $\Cl_n$ can be identified with $\L^{\bcdot}(\bR^n)$ in such a way that \beq v \!\cdot\! w = v \wedge w - v \lrcorner w \quad \text{ for every } \ v \in \bR^n \, , \; w \in \L^{\bcdot}(\bR^n) \ . \label{isoGrCl} \eeq The spin group is the subset $\Spin_n \= \big\{v_1 \!\cdot\! {\dots} \!\cdot\! v_{2r} : v_j \in \bR^n, \, ||v_j||=1\big\}$ equipped with the multiplication of $\Cl_n$. If $n \geq 3$, it is simply connected and it is the universal covering of $\SO_n$ by means of the map $$\t_n: \Spin_n \ra \SO_n \, , \quad \,\,\,\,\,\,\, \t_n(v_1\!\cdot\!\dots\!\cdot\!v_{2r}) \= \operatorname{refl}_{v_1} \circ {\dots} \circ \operatorname{refl}_{v_{2r}} \,\, ,$$ where $\operatorname{refl}_v$ is the reflection of $\bR^n$ with respect to $v^{\perp}$. We denote by $\eS_n$ the space of complex $n$-spinors, by $\bcdot: \Cl_n \otimes \eS_n \ra \eS_n$ the standard Clifford multiplication and by $\k_n: \Spin_n \ra \SU(\eS_n)$ the {\it spin representation of $\Spin_n$}, where we consider $\eS_n$ endowed with a positive definite Hermitian scalar product which is invariant under Clifford multiplications by vectors $v \in \bR^n \subset \Cl_n$.

Let now $\Cl_n^{\, \bC} = \Cl_n \!\otimes_{\bR} \bC$ be the complex Clifford algebra. In the even dimensional case, $\Cl_{2m}^{\, \bC}$ is generated by complex vectors $(\e_j, \overline{\e}_j)$, related with the generators $(e_j)$ by the formula \eqref{vettu}, which verify \beq \e_r \!\cdot\! \e_s + \e_s \!\cdot\! \e_r = \overline{\e}_r \!\cdot\! \overline{\e}_s + \overline{\e}_s \!\cdot\! \overline{\e}_r = 0 \, , \quad \e_r \!\cdot\! \overline{\e}_s + \overline{\e}_s \!\cdot\! \e_r = -2\d_{rs} \, .
\label{anticomeps} \eeq 

Finally, we recall that the {\it $\Spin^{\, \bC}$-group} is the Lie group $\Spin_n^{\, \bC} \= \Spin_n \times_{{}_{\bZ_2}} S^1$. It is a $2$-fold covering of $\SO_n \times S^1$ by means of the map $(\t_n, \varrho_n) : \Spin_n^{\, \bC} \ra \SO_n \times S^1$, where \begin{align*}
&\t_n: \Spin_n^{\, \bC} \ra \SO_n \,\, , \hskip 0.5cm \t_n([g,z]) \= \t_n(g) \,\, , \\
&\varrho_n: \Spin_n^{\, \bC} \ra S^1 \,\, , \hskip 0.5cm \varrho_n([g,z]) \= z^2 \,\, ,
\end{align*} and admits a representation on $\eS_n$, denoted again by $\k_n$, defined by \beq \k_n : \Spin_n^{\, \bC} \ra \SU(\eS_n) \ , \quad \k_n([g,z]) \= z \k_n(g) \ . \label{spincrep} \eeq

\subsection{$\spin^{\bC}$ structures on Hermitian manifolds}

Let $(M,g)$ be an oriented Riemannian manifold with oriented orthonormal frame bundle $\pi: \SO_g(M) \ra M$. A {\it $\spin^{\bC}$ structure on $(M,g)$} is a $\Spin^{\, \bC}_n$-bundle $\hat{\pi}: \eP \rightarrow M$ together with an equivariant bundle morphism $\varpi: \eP \ra \SO_g(M)$ such that $\hat{\pi}=\pi \circ \varpi$. Given a $\spin^{\bC}$ structure $\eP$, the corresponding {\it spinor bundle} is the associated bundle $\eS M \= \eP \times_{\k_n} \eS_n$. The space of its global sections is indicated with $\gS(M)$. \smallskip

Most, but not all, orientable Riemannian manifolds admit a $\spin^{\bC}$ structure (see \cite{LM}, p. 393). A crucial property of the subclass of Hermitian manifolds is that all of them have two very natural $\spin^{\bC}$ structures. In fact, the homomorphisms \beq f_{\pm}: \U_n \ra \SO_{2n} \times S^1 \ ,\quad f_{\pm}(A) \= \big(\imath_{{}_{\U_n}}(A),\, \det(A)^{\pm 1}\big) \ , \label{fpiumeno} \eeq where $\imath_{{}_{\U_n}}: \U_n \hookrightarrow \SO_{2n}$ is the canonical immersion of $\U_n$ into $\SO_{2n}$, can be uniquely lifted to two group homomorphisms $F_{\pm}: \U_n \ra \Spin_{2n}^{\, \bC}$ in such a way that the diagram \vspace {-5pt} {\scalefont{0.8} \begin{displaymath}\xymatrix{ 
 & & \Spin_{2n}^{\, \bC} \ar[d]^-{(\t_{2n}, \varrho_{2n})} \\
\U_n \ar[rr]^-{f_{\pm}} \ar[urr]^-{F_{\pm}} & & \SO_{2n}\times S^1
}\end{displaymath}}\vspace {-5pt}commutes.

\begin{definition} Let $(M,g,J)$ be an Hermitian $2n$-manifold. Its {\it canonical $\spin^{\bC}$ structure} is the bundle $\eP^{{}^\uparrow}\!(M) \= \U_{\!{}_{g,J}\!}(M) \times_{{}_{F_+}} \Spin_{2n}^{\, \bC}$. Similarly, its {\it anticanonical $\spin^{\bC}$ structure} is $\eP^{{}^\downarrow}\!(M) \= \U_{\!{}_{g,J}\!}(M) \times_{{}_{F_-}} \Spin_{2n}^{\, \bC}$. \label{canantican}\end{definition}

If $\eS M$ is a spinor bundle on $(M,g,J)$ associated with a $\spin^{\bC}$ structure, it is known that the K\"ahler form $\w=\w_{{}_{g,J}}$ acts on $\eS M$ by Clifford multiplication as a bundle endomorphism. Its eigenvalues are the imaginary numbers $(2k-n)i$, $0 \leq k \leq n$, and, in each fibre, the corresponding eigenspaces $$\eS_x^kM \= \{\psi \in \eS_xM : \w_x \!\bcdot\! \psi=(2k-n)i\,\psi \} \ , \quad 0 \leq k \leq n \ , \quad x \in M$$ have dimension $\binom{n}{k}$. One can also directly check that \beq \begin{split} \eS^0_xM=\left\{\psi \in \eS_xM : \overline{v} \!\bcdot\! \psi=0 \text{ for every } \overline{v} \in T_x^{01}M\right\} \ , \\ \eS^n_xM=\left\{\psi \in \eS_xM : v \!\bcdot\! \psi=0 \text{ for every } v \in T_x^{10}M\right\} \ . \end{split} \label{S0n} \eeq Furthermore, it is known that there exist Hermitian metrics on $\eS M$ invariant under Clifford multiplication by tangent vectors. With respect to one of such metrics, for every $0 \leq k \leq n$, the maps \beq \begin{gathered}
\a^k: \L^{0,k}(T^*M) \!\otimes\! \eS^0M \ra \eS^kM \ , \quad \a^k(\overline{\mu} \!\otimes\! \psi) \= \frac{1}{2^{\frac{k}{2}}} \overline{\mu} \!\bcdot\! \psi \,\, , \\
\b^k: \L^{k,0}(T^*M) \!\otimes\! \eS^nM \ra \eS^{n-k}M \ , \quad \b^k(\nu \!\otimes\! \psi) \= \frac{1}{2^{\frac{k}{2}}} \nu \!\bcdot\! \psi 
\end{gathered} \label{ab} \eeq are $\bC$-linear isometries and their sums give rise to global isometries $$\a: \L^{0,\bcdot}(T^*M) \otimes \eS^0 M \overset{\simeq}{\longrightarrow} \eS M \ , \quad \b: \L^{\bcdot,0}(T^*M) \otimes \eS^n M \overset{\simeq}{\longrightarrow} \eS M \ .$$

\section{Chern-Dirac bundles} \setcounter{equation} 0

\subsection{Chern-Dirac bundles and partial Chern-Dirac operators}

Given an Hermitian $2n$-manifold $(M,g,J)$, we can always consider the {\it complex Clifford bundle} over $M$, defined by $\Cl^{\,\bC}\!M \= \U_{\!{}_{g,J}\!}(M) \times_{{}_{\U_n}} \Cl_{2n}^{\, \bC}$, where the group $\U_n$ acts on $\Cl_{2n}^{\, \bC}$ in the standard way. It is the complex analogue of the (real) Clifford bundles considered in \cite{LM}, \S I.3. In full analogy with the notion of (real) {\it Dirac bundle} (see e.g. \cite{LM}, \S I.5), it is convenient to introduce the following

\begin{definition} A {\it Chern-Dirac bundle} over $(M,g,J)$ is a complex vector bundle $\pi: E \ra M$ endowed with an Hermitian metric $h$, a covariant derivative $D$ which preserves the metric and a structure of complex left $\Cl^{\,\bC}\!M$-module $c: \Cl^{\,\bC}\!M \ra \gg\gl(E)$, satisfying the conditions: \begin{itemize}
\item[i)] for every $v \in T^{\,\bC}M$, $\s_1, \s_2 \in E$ \beq h(c(v)\s_1 , \s_2)+h(\s_1 , c(\overline{v})\s_2)=0 \, ; \label{skewH}\eeq
\item[ii)] for every $X \in \gX(M)$, for every sections $w$ of $\Cl^{\,\bC}\!M$ and $\s$ of $E$ \beq D_X(c(w)\s)=c\big(\Derc{\cC}_Xw\big)\s+c(w)D_X\s \, . \label{derE}\eeq
\end{itemize} \label{CDbundles} \end{definition}

Note that, if $M$ is K\"ahler, then $\Derc{\cC}_{\color{white} X}\!\! = \Derc{LC}_{\color{white} X}\!\!$ and, consequently, any Chern-Dirac bundle over $M$ is a Dirac bundle in the usual sense. \smallskip

The main results of this paper are based on the following differential operators on Chern-Dirac bundles, which are natural analogues of Dirac operators.

\begin{definition} Let $E$ be a Chern-Dirac bundle over $(M,g,J)$. The {\it partial Chern-Dirac operators} on the section of $E$ are the maps $\Diracde\!'$ and $\Diracde\!''$ that transform any section $\s$ of $E$ into the sections defined for every $x \in M$ by \beq \begin{gathered}
\Diracde\!'\s\big|_x \= \sum_{j=1}^n c(\overline{\e}_j)D_{\e_j}\s-\frac12\sum_{r<s}c(\overline{\e}_r \!\cdot\! \overline{\e}_s \!\cdot\! T_{rs})\s_x \,\, ,\\
\Diracde\!''\s\big|_x \= \sum_{j=1}^n c(\e_j)D_{\overline{\e}_j}\s-\frac12\sum_{r<s}c(\e_r \!\cdot\! \e_s \!\cdot\! T_{\bar{r}\bar{s}})\s_x  \,\, ,
\end{gathered} \label{pCD} \eeq where $(\e_j, \overline{\e}_j)$ is the normalized complex basis \eqref{vettu} associated with a unitary basis $(e_j) \subset T_xM$, $T$ is the torsion of $\Derc{\cC}_{\color{white} X}$\!\! and $T_{rs}\= T(\e_r, \e_s)$, $T_{\bar{r}\bar{s}}\= T(\overline{\e}_r, \overline{\e}_s)$. The sum of these operators gives what we call {\it Chern-Dirac operator} \beq \slh{D} \= \Diracde\!' + \Diracde\!'' \, . \label{CD} \eeq \end{definition}

One can directly check that the formulas \eqref{pCD} define two global operators on the whole manifold, i.e. they are coordinate invariant. Indeed, if $(e'_j)$ and $(e_k)$ are unitary frames at $x \in M$ with $e'_j=e_kA^k_j$, then the associated normalized complex frames are related by \beq \e'_j=\e_k \a^k_j \,\, , \quad \overline{\e}'_j=\overline{\e}_k \overline{\a}^k_j \,\, , \quad \,\, \text{ with } \,\, \a^k_j \= A^{2k-1}_{2j-1}-iA^{2k-1}_{2j} \,\, .\eeq Since the complex coefficients $\a^k_j$ verify $\sum_j \a^k_j\,\overline{\a}^m_j= \d^{km}$, then \begin{align*}
\sum_{j=1}^n c(\overline{\e}'_j)D_{\e'_j}\s-&\frac12\sum_{r<s}c(\overline{\e}'_r \!\cdot\! \overline{\e}'_s \!\cdot\! T(\e'_r, \e'_s))\s_x = \\
&=\sum_{j=1}^n \overline{\a}^k_j\a^h_jc(\overline{\e}_k)D_{\e_h}\s-\frac12\sum_{r<s}\overline{\a}^{\ell}_r\a^p_r\overline{\a}^m_s\a^q_s \,c(\overline{\e}_{\ell} \!\cdot\! \overline{\e}_m \!\cdot\! T(\e_p, \e_q))\s_x \\
&=\sum_{k=1}^n c(\overline{\e}_k)D_{\e_k}\s-\frac12\sum_{\ell<m}c(\overline{\e}_{\ell} \!\cdot\! \overline{\e}_m \!\cdot\! T(\e_{\ell}, \e_m))\s_x
\end{align*} and analogously for $\Diracde\!''$. \smallskip

It follows from the definition that the Chern-Dirac operator $\slh{D}$ is a first-order elliptic operator. Moreover, it turns out that the operators \eqref{pCD} are formal adjoints of one another and, consequently, $\slh{D}$ is formal self-adjoint. In fact

\begin{prop} Let $E$ be a Chern-Dirac bundle over $M$ and $\s_1$, $\s_2$ be two sections of $E$. Then \beq h(\Diracde\!'\s_1,\s_2) - h(\s_1,\Diracde\!''\s_2) = \Div(V^{\s_1,\s_2}) \ , \label{diracdiv} \eeq where $V^{\s_1,\s_2}$ is the unique complex vector field that satisfies $$g(V^{\s_1,\s_2},X)=-h(\s_1,c(X^{10})\s_2) \quad \text{ for every } X \in \gX(M) \, .$$ Consequently, if $\s_1$, $\s_2$ are compactly supported, then $$\int_Mh(\Diracde\!'\s_1,\s_2) \,d\!\operatorname{vol}_g = \int_Mh(\s_1,\Diracde\!''\s_2) \,d\!\operatorname{vol}_g \,\, .$$ \label{adjoint} \end{prop}
\begin{proof} Let $(e_j)$ be a unitary frame field defined on some open subset $\cU \subset M$ and $\e_j$, $\overline{\e}_j$ the associated normalized complex vector fields defined in \eqref{vettu}. One can directly check that the vector field $V^{\s_1,\s_2}$ takes values in $T^{01}M$. By the properties of Levi-Civita connection, this implies that \begin{align}
\Div(V^{\s_1,\s_2}) &= \sum_j\Big(g(\Derc{LC}_{\e_j}V^{\s_1,\s_2}, \overline{\e}_j)+g(\Derc{LC}_{\overline{\e}_j}V^{\s_1,\s_2}, \e_j)\Big) \nonumber \\
&= \sum_j\!\Big({-}g(V^{\s_1,\s_2}, \Derc{LC}_{\e_j}\overline{\e}_j){+}\overline{\e}_j\big(g(V^{\s_1,\s_2},\e_j)\big){-}g(V^{\s_1,\s_2}, \Derc{LC}_{\overline{\e}_j}\e_j)\Big)  \label{divV} \\
&= \sum_j\Big(-\overline{\e}_j\big(h(\s_1,c(\e_j)\s_2)\big)-\Div(\overline{\e}_j) h(\s_1,c(\e_j)\s_2)\Big) \, . \nonumber
\end{align} On the other hand, we also have that \beq \sum_j \Derc{\cC}_{\overline{\e}_j}\e_j = \sum_j\big({-}\Div(\overline{\e}_j)+\q(\overline{\e}_j)\big)\e_j \, . \label{Dee} \eeq Using \eqref{divV} and \eqref{Dee} we obtain that \begin{align*}
h(\Diracde\!'\s_1,\s_2) &= \sum_j h(c(\overline{\e}_j)D_{\e_j}\s_1,\s_2)-\frac12\sum_{r<s}h(c(\overline{\e}_r)c(\overline{\e}_s)c(T_{rs})\s_1, \s_2)  \\
&= \sum_j -h(D_{\e_j}\s_1, c(\e_j)\s_2)+\frac12\sum_{r<s}h(\s_1, c(T_{\bar{r}\bar{s}})c(\e_s)c(\e_r)\s_1, \s_2) \\
&= \sum_j \Big(-\overline{\e}_j\big(h(\s_1, c(\e_j)\s_2)\big)+h(\s_1, D_{\overline{\e}_j}(c(\e_j)\s_2))\Big)+ \\
& \phantom{aaaaaaaaaaaaaaaaaaaaaaaaaaaaai}+\frac12\sum_{r<s}h(\s_1, c(T_{\bar{r}\bar{s}})c(\e_s)c(\e_r)\s_1, \s_2) \\
&= \sum_j \Big(-\overline{\e}_j\big(h(\s_1, c(\e_j)\s_2)\big)+h(\s_1,c\big(\Derc{\cC}_{\overline{\e}_j}\e_j\big)\s_2)+h(\s_1,c(\e_j)D_{\overline{\e}_j}\s_2)\Big)+ \\
& \phantom{aaaaaaaaaaaaaaaaaaaaaaaaaaaaai}+\frac12\sum_{r<s}h(\s_1, c(T_{\bar{r}\bar{s}})c(\e_s)c(\e_r)\s_1, \s_2) \\
&\overset{\eqref{Dee}}{=} \!\!\sum_j \!\!\Big(\!{-}\overline{\e}_j\!\big(h(\s_1, c(\e_j)\s_2)\big){-}\Div(\overline{\e}_j)h(\s_1,c(\e_j)\s_2){+}\q(\overline{\e}_j)h(\s_1,c(\e_j)\s_2){+} \\
& \phantom{aaaaaaaaaaaii}+h(\s_1,c(\e_j)D_{\overline{\e}_j}\s_2)\Big)+\frac12\sum_{r<s}h(\s_1, c(T_{\bar{r}\bar{s}})c(\e_s)c(\e_r)\s_1, \s_2) \\
&\overset{\eqref{divV}}{=}\Div(V^{\s_1,\s_2}) +\sum_j h(\s_1,c(\e_j)D_{\overline{\e}_j}\s_2)+\sum_j\q(\overline{\e}_j)h(\s_1,c(\e_j)\s_2)+ \\
& \phantom{aaaaaaaaaaaaaaaaaaaaaaaaaaaaai}+\frac12\sum_{r<s}h(\s_1, c(T_{\bar{r}\bar{s}})c(\e_s)c(\e_r)\s_1, \s_2) \\
&=\Div(V^{\s_1,\s_2}) {+}\sum_j h(\s_1,c(\e_j)D_{\overline{\e}_j}\s_2){-}\frac12 \sum_{r<s}h(\s_1, c(\e_r)c(\e_s)c(T_{\bar{r}\bar{s}})\s_2) \\
&=\Div(V^{\s_1,\s_2}) +h(\s_1,\Diracde\!''\s_2) \, ,
\end{align*} so that \eqref{diracdiv} holds. The last assertion follows from Stokes' Theorem. \end{proof}

\subsection{Bochner-type formulas for Chern-Dirac operators}

Let $E$ be a Chern-Dirac bundle over an Hermitian $2n$-manifold $(M,g,J)$.
We now determine Bochner-type formulas for the squares of $\Diracde\!'$, $\Diracde\!''$ and of the Chern-Dirac operator $\slh{D}$. For this, we have to introduce a few operators on sections of $E$, determined by the curvature and the torsion of Chern connection. \smallskip

First, we consider the action of the curvature $R$ of $D$ on sections $\s$ of $E$ $$R_{XY}\s \= D_X D_Y \s - D_Y D_X \s - D_{[X,Y]}\s \quad \text{ for every } X,Y \in \gX(M) \,\, .$$ Second, for each section $\s$, we define \beq \begin{gathered}
\cR^{2,0}\s\big|_x \= \sum_{j<k}c(\overline{\e}_j \!\cdot\! \overline{\e}_k)R_{\e_j \e_k}\s \, , \,\quad\, \cR^{0,2}\s\big|_x \= \sum_{j<k}c(\e_j \!\cdot\! \e_k)R_{\overline{\e}_j \overline{\e}_k}\s \, , \\
\cR^{1,1}\s\big|_x \= \sum_{j \neq k}c(\overline{\e}_j \!\cdot\! \e_k)R_{\e_j \overline{\e}_k}\s \, , \quad\, \cR\s \= \cR^{2,0}\s +\cR^{1,1}\s +\cR^{0,2}\s \,\, ,
\end{gathered} \label{cR} \eeq \beq \begin{gathered}
\cT_1\s\big|_x \= \sum_{\substack{j<k \\ r<s}} c\big(\e_j \!\cdot\! \e_k \!\cdot\! \overline{\e}_r \!\cdot\! \overline{\e}_s \!\cdot\! T_{\bar{j}\bar{k}} \!\cdot\! T_{rs} + \overline{\e}_j \!\cdot\! \overline{\e}_k \!\cdot\! \e_r \!\cdot\! \e_s \!\cdot\! T_{jk} \!\cdot\! T_{\bar{r}\bar{s}}\big)\s \,\, , \label{cT}\\
\cT_2\s\big|_x \= \sum_{j \neq k} c\big(\e_j \!\cdot\! (\Derc{\cC}_{\e_k}T)_{\bar{j}\bar{k}}+\overline{\e}_j \!\cdot\! (\Derc{\cC}_{\overline{\e}_k}T)_{jk}\big)\s \,\, , \end{gathered} \eeq where $(\e_j, \overline{\e}_j)$ is the usual normalized complex frame \eqref{vettu} determined by a unitary frame $(e_j)$ for $T_xM$. Third, we define as $\cQ$ the first order differential operator on sections of $E$ by \beq \cQ\s\big|_x \= \sum_{j \neq k} \Big(c(\e_k \!\cdot\! T_{\bar{j}\bar{k}})D_{\e_j}\s+c(\overline{\e}_k \!\cdot\! T_{jk})D_{\overline{\e}_j}\s\Big) \,\, . \label{cQ}\eeq A first Bochner-type formula is the following.

\begin{theo} On each Chern-Dirac bundle $E$ over $M$, we have \beq (\Diracde\!')^2= \cR^{2,0} \, , \quad (\Diracde\!'')^2= \cR^{0,2} \,\, . \label{Rsquare} \eeq \label{squarepCD} \end{theo}
\begin{proof} Consider the decompositions of the partial Chern-Dirac operators into sums of differential operators of order $1$ and $0$, namely \beq \Diracde\!' = A'-\frac12B' \,\, , \quad \Diracde\!'' = A''-\frac12B'' \,\, , \eeq where $$A'\!{\=} \!\!\sum_k\!c(\overline{\e}_k)D_{\e_k} \, , \,\,\, B' \!{\=} \!\!\sum_{r<s}\!c(\overline{\e}_r \!\cdot\! \overline{\e}_s \!\cdot\! T_{rs}) \, , \,\,\, A''\!{\=} \!\!\sum_j\!c(\e_j)D_{\overline{\e}_j} \, , \,\,\, B'' \!{\=} \!\!\sum_{r<s}\!c(\e_r \!\cdot\! \e_s \!\cdot\! T_{\bar{r}\bar{s}}) \, .$$ Using \eqref{anticomeps} and standard properties of metric connections, with some tedious but straightforward computations, one gets that for each section $\s$ of $E$ \begin{itemize}
\item[i)] $\displaystyle \big(A'\big)^2\s = - \sum_{j<k}c(\overline{\e}_j \!\cdot\! \overline{\e}_k)D_{T_{jk}}\s +\cR^{2,0}\s $ ;
\item[ii)] $\displaystyle A'B'\s = -\sum_{\substack{j<k \\ m}}c\big(\overline{\e}_j \!\cdot\! \overline{\e}_k \!\cdot\! \overline{\e}_m \!\cdot\! T([\e_j, \e_k], \e_m)\big)\s -\sum_{\substack{j<k \\ m}}c\big(\overline{\e}_j \!\cdot\! \overline{\e}_k \!\cdot\! \overline{\e}_m \!\cdot\! T(T_{jk}, \e_m)\big)\s +$
\item[] $\hfill \displaystyle +\sum_{\substack{j<k \\ m}}c\big(\overline{\e}_j \!\cdot\! \overline{\e}_k \!\cdot\! \overline{\e}_m \!\cdot\! \Derc{\cC}_{\e_m}T_{jk}\big)\s+\sum_{\substack{j<k \\ m}}c\big(\overline{\e}_j \!\cdot\! \overline{\e}_k \!\cdot\! \overline{\e}_m \!\cdot\! T_{jk}\big)D_{\e_m}\s$ ;
\item[iii)] $\displaystyle B'A'\s = \sum_{\substack{j<k \\ m}} c\big(\overline{\e}_j \!\cdot\! \overline{\e}_k \!\cdot\! T_{jk} \!\cdot\! \overline{\e}_m\big)D_{\e_m}\s$ ;
\item[iv)] $\displaystyle \big(B'\big)^2\s = -2 \sum_{\substack{j<k \\ m}}c\big(\overline{\e}_j \!\cdot\! \overline{\e}_k \!\cdot\! \overline{\e}_m \!\cdot\! T(T_{jk}, \e_m)\big)\s$ .
\end{itemize} With similar computations, one also get that \beq \sum_{\substack{j<k \\ m}}c\big(\overline{\e}_j \!\cdot\! \overline{\e}_k \!\cdot\! \overline{\e}_m \!\cdot\! T_{jk}\big)D_{\e_m}\s + \sum_{\substack{j<k \\ m}}c\big(\overline{\e}_j \!\cdot\! \overline{\e}_k \!\cdot\! T_{jk} \!\cdot\! \overline{\e}_m\big)D_{\e_m}\s = -2\sum_{j<k}c(\overline{\e}_j \!\cdot\! \overline{\e}_k)D_{T_{jk}}\s \,\, .\label{checimetto} \eeq From (i) -- (iv) and \eqref{checimetto}, it follows easily that \begin{align*}
(\Diracde\!')^2\s \!&= (A'\big)^2\s-\frac12\big(A'B'+B'A'\big)\s+\frac14\big(B'\big)^2\s \\
&= \cR^{2,0}\s - \sum_{j<k}c(\overline{\e}_j \!\cdot\! \overline{\e}_k)D_{T_{jk}}\s +\frac12 \sum_{\substack{j<k \\ m}}c\big(\overline{\e}_j \!\cdot\! \overline{\e}_k \!\cdot\! \overline{\e}_m \!\cdot\! T([\e_j, \e_k], \e_m)\big)\s + \\
&\phantom{aaaaaaaaaaaaa} {+} \frac12 \sum_{\substack{j<k \\ m}}c\big(\overline{\e}_j \!\cdot\! \overline{\e}_k \!\cdot\! \overline{\e}_m \!\cdot\! T(T_{jk}, \e_m)\big)\s -\frac12 \sum_{\substack{j<k \\ m}}c\big(\overline{\e}_j \!\cdot\! \overline{\e}_k \!\cdot\! \overline{\e}_m \!\cdot\! \Derc{\cC}_{\e_m}T_{jk}\big)\s- \\
&\phantom{aaaaaaaaaaaaaaai} {-}\frac12 \sum_{\substack{j<k \\ m}}c\big(\overline{\e}_j \!\cdot\! \overline{\e}_k \!\cdot\! \overline{\e}_m \!\cdot\! T_{jk}\big)D_{\e_m}\s -\frac12 \sum_{\substack{j<k \\ m}}c\big(\overline{\e}_j \!\cdot\! \overline{\e}_k \!\cdot\! T_{jk} \!\cdot\! \overline{\e}_m\big)D_{\e_m}\s - \\
&\phantom{aaaaaaaaaaaaaaaaaaaaaaaaaaaaaaaaaaaaaaa} {-}\frac12 \sum_{\substack{j<k \\ m}}c\big(\overline{\e}_j \!\cdot\! \overline{\e}_k \!\cdot\! \overline{\e}_m \!\cdot\! T(T_{jk}, \e_m)\big)\s \\
&=\cR^{2,0}\s+\frac12 \sum_{\substack{j<k \\ m}}c\big(\overline{\e}_j \!\cdot\! \overline{\e}_k \!\cdot\! \overline{\e}_m \!\cdot\! \big(T([\e_j, \e_k], \e_m)- \Derc{\cC}_{\e_m}T_{jk} \big)\big)\s \\
&=\cR^{2,0}\s+\frac12 \!\!\sum_{j<k<m}\!\!\!c\big(\overline{\e}_j \!\cdot\! \overline{\e}_k \!\cdot\! \overline{\e}_m \!\cdot\! \big(T([\e_j, \e_k], \e_m)+T([\e_k, \e_m], \e_j)+ \\
&\phantom{aaaaaaaaaaaaaaaaaaaaaai} +T([\e_m, \e_j], \e_k)- \Derc{\cC}_{\e_m}T_{jk}- \Derc{\cC}_{\e_j}T_{km}- \Derc{\cC}_{\e_k}T_{mj}\big)\big)\s \, .
\end{align*} Here, the second term vanishes because of the First Bianchi Identity. Indeed \begin{align*}
\Derc{\cC}_{\e_j}&T_{km}+ \Derc{\cC}_{\e_m}T_{jk}+ \Derc{\cC}_{\e_k}T_{mj} = \\
&= \big(\Derc{\cC}_{\e_j}T\big)(\e_k,\e_m)+\big(\Derc{\cC}_{\e_k}T\big)(\e_m,\e_j)+\big(\Derc{\cC}_{\e_m}T\big)(\e_j,\e_k) + T(\Derc{\cC}_{\e_j}\e_k, \e_m)+ \\
&\phantom{aai} +T(\e_k, \Derc{\cC}_{\e_j}\e_m) + T(\Derc{\cC}_{\e_k}\e_m, \e_j)+ T(\e_m, \Derc{\cC}_{\e_k}\e_j) +T(\Derc{\cC}_{\e_m}\e_j, \e_k)+ T(\e_j, \Derc{\cC}_{\e_m}\e_k) \\
&= -T(T_{jk}, \e_m)-T(T_{km}, \e_j)-T(T_{mj}, \e_k)+T(\Derc{\cC}_{\e_j}\e_k - \Derc{\cC}_{\e_k}\e_j, \e_m)+ \\
&\phantom{aaaaaaaaaaaaaaaaaaaaaaaaaa} +T(\Derc{\cC}_{\e_k}\e_m - \Derc{\cC}_{\e_m}\e_k, \e_j)+T(\Derc{\cC}_{\e_m}\e_j - \Derc{\cC}_{\e_j}\e_m, \e_k) \\
&= T([\e_j, \e_k], \e_m)+T([\e_k, \e_m], \e_j)+T([\e_m, \e_j], \e_k) \, .
\end{align*} This proves that $(\Diracde\!')^2= \cR^{2,0}$. The identity $(\Diracde\!'')^2= \cR^{0,2}$ is proven similarly. \end{proof}

A formula for the square of the Chern-Dirac operator is given as follows.

\begin{theo} The Chern-Dirac operator of $E$ verifies \beq \slh{D}^2 = \D + \cQ + \cR + \frac14\cT_1 - \frac12\cT_2 \,\, , \label{Bform1} \eeq where $\D$ is the rough Laplacian of $D$ defined locally by $$\D\s \= -\sum_{j=1}^{2n}\big(D_{e_j}D_{e_j}\s-D_{\Derc{\cC}_{e_j}e_j}\s\big)$$ and $\cQ$, $\cR$, $\cT_1$, $\cT_2$ are the operators defined in \eqref{cR} -- \eqref{cQ}. \label{squareCD} \end{theo}
\begin{proof} As in the previous proof, $(e_j)$ is a locally defined unitary frame field and $(\e_j,\overline{\e}_j)$ the corresponding normalized complex frame field. Consider the decompositions of $\slh{D}$ into a sum of differential operators of order $1$ and $0$, respectively, namely \beq \slh{D}= A-\frac12B \,\, , \quad \text{ where } \,\,\,\, A\!{\=} \!\!\sum_k\!c(e_k)D_{e_k} \, , \,\,\, B\!{\=} \!\!\sum_{r<s}\!c(e_r \!\cdot\! e_s \!\cdot\! T_{e_r e_s}) \,\, . \eeq Then, with computations very similar to those of the previous proof, we get \begin{itemize}
\item[$i)$] $\displaystyle A^2\s = \D\s + \cR\s - \sum_{j<k}c(e_j \!\cdot\! e_k) D_{T_{e_j e_k}}\s$ ;
\item[$ii)$] $\displaystyle (AB{+}BA)\s \!= {-}\!\sum_{\substack{j<k \\ m}}\!c\big(e_j \!\cdot\! e_k \!\cdot\! e_m \!\cdot\! T([e_j, e_k], e_m)\big)\s {-}\!\sum_{\substack{j<k \\ m}}\!c\big(e_j \!\cdot\! e_k \!\cdot\! e_m \!\cdot\! T(T_{e_j e_k}, e_m)\big)\s +$
\item[] $\hfill \displaystyle +\sum_{\substack{j<k \\ m}}c\big(e_m \!\cdot\! e_j \!\cdot\! e_k \!\cdot\! \Derc{\cC}_{e_m}T_{e_j e_k}\big)\s-2\sum_{j<k}c(e_j \!\cdot\! e_k)D_{T_{e_j e_k}}\s +$
\item[] $\hfill \displaystyle +\sum_{r,s}c\big(e_s \!\cdot\! T(\Derc{\cC}_{e_r}e_r,e_s)\big)\s-2\sum_{r,s}c\big(e_s \!\cdot\! T_{e_r e_s}\big)D_{e_r}\s$ ;
\item[$iii)$] $\displaystyle B^2\s = \sum_{\substack{j<k \\ r<s}}c\big(e_j \!\cdot\! e_k \!\cdot\! e_r \!\cdot\! e_s \!\cdot\! T_{e_j e_k} \!\cdot\! T_{e_r e_s}\big)\s -2 \sum_{\substack{j<k \\ m}}c\big(\overline{\e}_j \!\cdot\! \overline{\e}_k \!\cdot\! \overline{\e}_m \!\cdot\! T(T_{jk}, \e_m)\big)\s$ .
\end{itemize} Combining $(i)$, $(ii)$ and $(iii)$ we have \begin{align*}
\slh{D}^2\s &= A^2\s-\frac12(AB+BA)\s+\frac14B^2\s \\
&= \D\s{+}\cR\s{+}\frac12\sum_{\substack{j<k \\ m}}\!c\big(e_j \!\cdot\! e_k \!\cdot\! e_m \!\cdot\! T([e_j, e_k], e_m)\big)\s{-}\frac12\sum_{\substack{j<k \\ m}}c\big(e_m \!\cdot\! e_j \!\cdot\! e_k \!\cdot\! \Derc{\cC}_{e_m}T_{e_j e_k}\big)\s{-} \\
&\phantom{aaaaaaaaaaaaaaaaaaaaaaa}-\frac12\sum_{r,s}c\big(e_s \!\cdot\! T(\Derc{\cC}_{e_r}e_r,e_s)\big)\s+\sum_{r,s}c\big(e_s \!\cdot\! T_{e_r e_s}\big)D_{e_r}\s+ \\
&\phantom{aaaaaaaaaaaaaaaaaaaaaaaaaaaaaaaaaaaaa}+\frac14\sum_{\substack{j<k \\ r<s}}c\big(e_j \!\cdot\! e_k \!\cdot\! e_r \!\cdot\! e_s \!\cdot\! T_{e_j e_k} \!\cdot\! T_{e_r e_s}\big)\s \\
&= \D\s + \cR\s + \frac12\sum_{r,s} c\Big(e_s \!\cdot\! \big(\Derc{\cC}_{e_r}T_{e_r e_s} + T([e_r, e_s], e_r) -T(\Derc{\cC}_{e_r}e_r,e_s)\big)\Big)\s + \\
&\phantom{aaaaaaaaaaaaaaaaaaaaaaaaaaaaaaaaaaaaaaaaaaaaaaaaaaaaaa} + \cQ\s + \frac14\cT_1\s \\
&= \D\s + \cR\s + \cQ\s + \frac14\cT_1\s + \frac12\sum_{r,s} c\big(e_s \!\cdot\! (\Derc{\cC}_{e_r}T)_{e_r e_s}\big)\s \\
&= \D\s + \cR\s + \cQ\s + \frac14\cT_1\s -\frac12\cT_2\s \,\, .
\end{align*} \end{proof}

The formula \eqref{Bform1} contains the first-order term $\cQ$, hence is hard to handle. For this reason, it is convenient to define a new covariant derivative $\hat{D}$ on $E$ by $$\hat{D}_X\s \= D_X\s -\frac12\sum_j c\big(e_j \!\cdot\! T(X,e_j)\big)\s \,\, .$$ With a straightforward computation, one can directly check that the rough Laplacian $\hat{\D}$ of $\hat{D}$, defined locally by $$\hat{\D}\s \= -\sum_{j=1}^{2n}\big(\hat{D}_{e_j}\hat{D}_{e_j}\s-\hat{D}_{\Derc{\cC}_{e_j}e_j}\s\big) \,\, ,$$ satisfies $$\hat{\D}=\D+\cQ-\frac12\cT_2+\frac14\sum_{j,k,m}c(e_j \!\cdot\! e_k \!\cdot\! T_{e_m,e_j} \!\cdot\! T_{e_m,e_k})$$ and so we obtain the following

\begin{theo} The Chern-Dirac operator of $E$ verifies \beq \slh{D}^2 = \hat{\D} + \cR-\frac12P-\frac18|T|^2 \,\, , \eeq where $P$ is defined by $$P \= \!\! \sum_{j<k<r<s} \!\!\!\! \big(g(T_{e_j,e_k},T_{e_r,e_s}) {+} g(T_{e_j,e_s},T_{e_k,e_r}) {+} g(T_{e_j,e_r},T_{e_k,e_s})\big) c(e_j \!\cdot\! e_k \!\cdot\! e_r \!\cdot\! e_s)$$ and $|T|^2=\sum_{j,k,r}T(e_j,e_k,e_r)^2$. \label{squareCD2} \end{theo}

Moreover, if the complex dimension of $M$ is $n=2$, then $P=0$ and $|T|^2=2|\q|^2$. So, we obtain the following

\begin{cor}The Chern-Dirac operator over a complex surface verifies \beq \slh{D}^2 = \hat{\D} + \cR-\frac14|\q|^2 \,\, . \eeq \label{squareCD3} \end{cor}

Following the same arguments of classical Bochner type theorems, Theorem \ref{squarepCD} and \ref{squareCD2} can be used to determine vanishing properties for solutions of equations of the form $\Diracde\!'\s=0$ or $\slh{D}\s=0$, provided that appropriate conditions on curvature and torsion are imposed. We do not investigate here such possibilities. \smallskip

In the next sections, we prove the existence of some important Chern-Dirac bundles, canonically associated with any Hermitian manifold (not necessarily K\"ahler), to which all results determined so far apply immediately.

\section{$\eV$-spinors and cohomology of Hermitian manifolds}

\subsection{Canonical and Anticanonical spinor bundle on Hermitian manifolds}

Let $(M,g,J)$ be an Hermitian $2n$-manifold. We indicate with $\eS^{{}^\uparrow}\!M$ and $\eS^{{}^\downarrow}\!M$ the spinor bundles on $M$ associated with the canonical and anticanonical $\spin^{\bC}$ structures, respectively. In these cases, one can directly check that the eigen-subbundles $\eS^{{}^\uparrow 0}M$, $\eS^{{}^\downarrow n}M$ are trivial line bundles. Hence, we may fix: \begin{itemize}
\item[-] two nowhere vanishing global sections \beq \psi^{\zero}: M \ra \eS^{{}^\uparrow 0}M \, , \quad \f^{\zero}: M \ra \eS^{{}^\downarrow n}M \, ; \label{fipsi} \eeq
\item[-] two Hermitian metrics $h^{{}^\uparrow}$, $h^{{}^\downarrow}$ on $\eS^{{}^\uparrow}M$, $\eS^{{}^\downarrow}M$, respectively, which are invariant under the Clifford multiplication by tangent vectors and such that \beq h^{{}^\uparrow}(\psi^{\zero},\psi^{\zero}) = 1 = h^{{}^\downarrow}(\f^{\zero},\f^{\zero}) \,\, . \label{huphdown} \eeq
\end{itemize} Tensoring with the sections $\psi^{\zero}$, $\f^{\zero}$, we may identify $\L^{0,q}(T^*M) \simeq \L^{0,q}(T^*M) {\otimes} \eS^{{}^\uparrow 0}M$ and $\L^{p,0}(T^*M) \simeq \L^{p,0}(T^*M) {\otimes} \eS^{{}^\downarrow n}M$, so that the maps \eqref{ab} determine isometries \beq \begin{gathered}
\a^{{}^\uparrow k}: \L^{0,k}(T^*M) \simeq \L^{0,k}(T^*M) \!\otimes\! \eS^{{}^\uparrow 0}M \ra \eS^{{}^\uparrow k}M \,\, , \\
\b^{{}^\downarrow k}: \L^{k,0}(T^*M) \simeq \L^{k,0}(T^*M) \!\otimes\! \eS^{{}^\downarrow n}M \ra \eS^{{}^\downarrow n-k}M \,\, .
\end{gathered} \label{abupdown} \eeq Moreover, from \eqref{relClif}, \eqref{isoGrCl} and \eqref{S0n}, one can directly prove the following useful

\begin{lemma} Let $\l$ be a $1$-form. Then for every $\nu \in \W^{p,0}(M)$, $\overline{\mu} \in \W^{0,q}(M)$ we have \beq \begin{split}
\l \!\cdot\! \overline{\mu} \!\bcdot\! \psi^{\zero} &= (\l^{01}\!\wedge\! \overline{\mu}) \!\bcdot\! \psi^{\zero} - 2\big((\l^\sharp)^{01} \lrcorner \overline{\mu} \big) \!\bcdot\! \psi^{\zero} \, ,\\
\l \!\cdot\! \nu \!\bcdot\! \f^{\zero} &= (\l^{10}\!\wedge\! \nu) \!\bcdot\! \f^{\zero} - 2\big((\l^\sharp)^{10} \lrcorner \nu \big) \!\bcdot\! \f^{\zero} \, .
\end{split} \label{prod1form} \eeq \end{lemma}

Considering the action of the K\"ahler form $\w$ on the dual bundle $\eS^{{}^\uparrow*}M$ given by $(\w \!\bcdot\! L)(\psi) \= -L(\w \!\bcdot\! \psi)$, we get a corresponding split $\eS^{{}^\uparrow*}M = \eS^{{}^\uparrow *0}M \!\oplus\! {\dots}\!\oplus\! \eS^{{}^\uparrow *n}M$. Note finally that the $\bC$-linear maps $$\d^k: \eS^{{}^\downarrow n-k}M \ra \eS^{{}^\uparrow *k}M \ , \quad \d^k(\f) \= h^{{}^\uparrow} \Big(\a^{{}^\uparrow k}\Big(\overline{(\b^{{}^\downarrow k})^{-1}(\f)}\Big)\ ,\, \cdot \,\Big)$$ are actually isometries. Moreover, if $\f = \nu \!\bcdot\! \f^{\zero} \in \eS^{{}^\downarrow n-k}M$ and $\psi = \overline{\mu} \!\bcdot\! \psi^{\zero} \in \eS^{{}^\uparrow k}M$, then we have that \beq \f(\psi) \= \d^k(\f)(\psi) = h^{{}^\uparrow}(\overline{\nu} \!\bcdot\! \psi^{\zero},\overline{\mu} \!\bcdot\! \psi^{\zero}) = g(\nu, \overline{\mu}) \,\, . \label{isodelta} \eeq From now on, we tacitly use such maps to identify $\eS^{{}^\downarrow}M \simeq \eS^{{}^\uparrow*}M$. \smallskip

We denote by $\gS^{{}^{\uparrow}}\!(M)$ and $\gS^{{}^{\downarrow}}\!(M)$ the spaces of global sections of $\eS^{{}^\uparrow}\!M$ and $\eS^{{}^\uparrow}\!M$, respectively. Since the unitary frame bundle $\U_{\!{}_{g,J}\!}(M)$ is a $\U_n$-reduction of both $\eP^{{}^\downarrow}\!(M)$ and $\eP^{{}^\uparrow}\!(M)$ (see Definition \ref{canantican}), we obtain the following

\begin{prop} The Hermitian bundles $(\eS^{{}^\uparrow}\!M, h^{{}^\uparrow})$, $(\eS^{{}^\downarrow}\!M, h^{{}^\downarrow})$ are Chern-Dirac bundles, with respect to the action of the Chern connection of $M$ and the standard Clifford multiplication, which are isometric to $\L^{0,\bcdot}(T^*M)$ and $\L^{\bcdot,0}(T^*M)$, respectively, by the maps \eqref{abupdown}. \label{SeCD} \end{prop}

\begin{proof} From the very definition of $\eP^{{}^\uparrow}\!(M)$, it follows that $$\eS^{{}^\uparrow}\!M = \U_{\!{}_{g,J}\!}(M) \times_{{}_{(\k_{2n} \circ F_+)}} \eS_{2n}$$ and so the Chern connection $\w^{{}^{{}_{\cC}}}\!$ on $\U_{\!{}_{g,J}\!}(M)$ defines a covariant derivative of the sections of $\eS^{{}^\uparrow}\!M$ by \beq \Derc{\cC}_{X}\psi \= d^{{}^{{}_{\,\cC}}}\!\psi(\widehat{X}) = d\psi(\widehat{X}) + (\k_{2n} \circ F_+)_*\big(\w^{{}^{{}_{\cC}}}(\widehat{X})\big)\psi \,\, , \label{Dpsi} \eeq where: \par

-- $\psi \in \gS^{{}^{\uparrow}}\!(M)$ is identified with a function $\psi: \U_{\!{}_{g,J}\!}(M) \ra \eS_{2n}$ such that $$\psi(uA)=\k_{2n}\big(F_+(A^{-1})\big)\psi(u) \quad \text{ for every } A \in \U_n \,\, ;$$

-- $\widehat{X}$ is the horizontal lift of $X$ on $T\U_{\!{}_{g,J}\!}(M)$ determined by the connection form $\w^{{}^{{}_{\cC}}}$;

-- $d^{{}^{{}_{\,\cC}}}\!$ is the exterior covariant derivative on $\U_{\!{}_{g,J}\!}(M)$ determined by $\w^{{}^{{}_{\cC}}}$.

Using \eqref{Dpsi} and the fact that the differential of $F_+$ is just $$(F_+)_*: \gu_n \ra \so_{2n} \oplus i\bR \,\, , \quad (F_+)_*(A) = (\t_{2n})_*^{-1}(A)+\frac12\Tr(A) \,\, ,$$ we get that $$\Derc{\cC}_{X}\big(Y \!\bcdot\!\psi\big)= \big(\Derc{\cC}_{X}Y\big) \!\bcdot\!\psi + Y \!\bcdot\!\Derc{\cC}_{X}\psi \quad \text{ for every } X, Y \in \gX(M) ,\, \psi \in \gS^{{}^{\uparrow}}\!(M) \,\, .$$ Since the representation $\k_{2n}$ of $\Spin_{2n}^{\,\bC}$ is unitary, it follows that $\Derc{\cC}_{\color{white} X}\!h^{{}^\uparrow}=0$ and, furthermore, one can directly check that \eqref{skewH} holds. The same arguments, mutatis mutandis, determine a covariant derivative $\Derc{\cC}_{\color{white} X}\!\!$ on spinors $\f \in \gS^{{}^{\downarrow}}\!(M)$ which fulfill the required conditions. \end{proof}

\subsection{A fundamental example of Chern-Dirac bundle: the $\eV$-spinors}

Given an Hermitian $2n$-manifold $(M,g,J)$, with canonical and anticanonical spinor bundles $\eS^{{}^\uparrow}\!M$, $\eS^{{}^\downarrow}\!M$, we define as {\it $\eV$-spinor bundle of $M$} the vector bundle $$\eV M \= \eS^{{}^\downarrow}\!M \!\otimes\! \eS^{{}^\uparrow}\!M \ .$$ We remark that $\eV M$ is equipped with the Hermitian metric $\check{h} \= h^{{}^\downarrow} \!\otimes\! h^{{}^\uparrow}$, where $h^{{}^\downarrow}$ and $h^{{}^\uparrow}$ are defined in \eqref{huphdown}, and with the bigradation given by the subbundles $$\eV^{p,q}M \= \eS^{{}^\downarrow n-p}M \!\otimes\! \eS^{{}^\uparrow q}M \simeq \eS^{{}^\uparrow *p}M \!\otimes\! \eS^{{}^\uparrow q}M \ , \quad 0 \leq p,q \leq n \ .$$ Note that, since $\eS^{{}^\uparrow 0}M$ and $\eS^{{}^\downarrow n}M$ are trivial, the subbundles $\eV^{\,\bcdot,0}M$ and $\eV^{\,0, \bcdot}M$ are isomorphic to $\eS^{{}^\downarrow}\!M$ and $\eS^{{}^\uparrow}\!M$, respectively, so that $\eV M$ can be considered as a bundle which naturally includes both the canonical and anticanonical spinor bundle. Let us denote the space of global sections of $\eV M$ by $\gV(M)$. \smallskip

From the proof of Proposition \ref{SeCD}, we get that the Chern connection defines a covariant derivative along vector fields of $M$ of the sections of $\eV M$ $$\Derc{\cC}_{\color{white} X}\!\!: \gX(M) \otimes \gV(M) \ra \gV(M) \, .$$

Let $\f^{\zero}$, $\psi^{\zero}$ be the global sections in \eqref{fipsi}, satisfying \eqref{huphdown}, and $\x^{\zero}$ the distinguished section $\x^{\zero} \= \f^{\zero} \otimes \psi^{\zero} \in \gV(M)$. Clearly $\x^{\zero}$ is a nowhere vanishing global section of $\eV^{0,0}M$ such that $\check{h}(\x^{\zero},\x^{\zero})=1$. Note that the action of $\Derc{\cC}_{\color{white} X}\!\!$ on $\gS^{{}^{\uparrow}0}(M) \simeq \Cinf(M;\bC)$ coincides with the Chern covariant derivative of the trivial holomorphic Hermitian line bundle $\big(\eS^{{}^{\uparrow}0}M,\, h^{{}^{\uparrow}}\big|_{\eS^{{}^{\uparrow}0}M \otimes \eS^{{}^{\uparrow}0}M},\psi^{\zero}\big)$ and therefore correspond to the trivial covariant derivative on this bundle. The same holds on $\gS^{{}^{\downarrow}n}(M) \simeq \Cinf(M;\bC)$. So, it follows that the action of $\Derc{\cC}_{\color{white} X}\!\!$ on sections of $\eV^{0,0}M$ is trivial too. \smallskip

The interest for the bundle of $\eV$-spinors comes from the fact that there are {\it two structures of complex left $\Cl^{\, \bC}\!M$-module on $\eV M$ which make it is a Chern-Dirac bundle isomorphic with the bundle of complex forms of $M$.} To see this, let \beq \bcdot_{\!\!{}_L} \ , \ \bcdot_{\!\!{}_R} : \Cl^{\, \bC}\!M \!\otimes\! \eV M \ra \eV M \label{Cliffmult} \eeq be the unique $\bC$-linear operations which transform each pair, given by an element $w \in \Cl_x^{\,\bC}M$ and a homogeneous decomposable $\eV$-spinor $\f {\otimes} \psi \in \eV^{p,q}_xM$ at some $x \in M$, into the $\eV$-spinors $$w \!\bcdot_{\!\!{}_L} (\f \!\otimes\! \psi) \!\=\! (w \!\bcdot\! \f) \!\otimes\! \psi \ , \quad w \!\bcdot_{\!\!{}_R} (\f \!\otimes\! \psi) \!\=\! ({-}1)^p\f \!\otimes\! (w \!\bcdot\! \psi) \ .$$ As mentioned above, the following properties hold.

\begin{prop} \hfil \par \begin{itemize}
\item[i)] The bundle $\eV M$ is a Chern-Dirac bundle with respect to both \eqref{Cliffmult}. 
\item[ii)] There exists an isometry $\vars: \L^{\bcdot}(T^{*\bC}M) \ra \eV M$ such that $\Derc{\cC}_X \circ \vars = \vars \circ \Derc{\cC}_X$ for every $X \in \gX(M)$. \end{itemize} \label{propVM} \end{prop}

\begin{proof} The first claim follows directly from Proposition \ref{SeCD}. For the second one, let $$\jmath^{\, p,q}: \Big(\L^{p,0}(T^*M) \!\otimes\! \eS^{{}^\downarrow n}M\Big) \!\otimes\! \Big(\L^{0,q}(T^*M) \!\otimes\! \eS^{{}^\uparrow 0}M\Big) \ra \L^{p,q}(T^*M) \!\otimes\! \eV^{0,0}M$$ be the vector bundle isomorphism which transforms decomposable elements into $$(\jmath^{\, p,q})_x \big((\nu \!\otimes\! \f^{\zero}_x) \!\otimes\! (\overline{\mu} \!\otimes\! \psi^{\zero}_x)\big) \= (\nu \wedge \overline{\mu}) \!\otimes\! \x^{\zero}_x \ , \quad x \in M \ .$$ Let also \beq \vars^{p,q}: \L^{p,q}(T^*M) \simeq \L^{p,q}(T^*M) \!\otimes\! \eV^{0,0}M \ra \eV^{p,q}M \ , \quad \vars^{p,q} \= \big(\b^{{}^\downarrow p} \!\otimes\! \a^{{}^\uparrow q}\big) \circ (\jmath^{\, p,q})^{-1} \ , \label{ropq} \eeq where $\a^{{}^\uparrow q}$ and $\b^{{}^\downarrow p}$ are the isometries defined in \eqref{abupdown}. By construction, each isomorphism $\vars^{p,q}$ is actually an isometry of Hermitian bundles and they all combine into a global isometry $$\vars: \L^{\bcdot}(T^{*\bC}M) \simeq \L^{\bcdot}(T^{*\bC}M) \otimes \eV^{0,0}M \ra \eV M \ .$$ The fact that $\Derc{\cC}_X \circ \vars = \vars \circ \Derc{\cC}_X$ follows from the previously mentioned property $\Derc{\cC}_{\color{white}X}\!\!\psi^{\zero}=\Derc{\cC}_{\color{white}X}\!\!\f^{\zero}=0$, which yields that for every $\h= \nu \wedge \overline{\mu} \in \W^{p,q}(M)$ and $X \in \gX(M)$ we have \begin{align*}
\Derc{\cC}_X \vars^{p,q}(\h) &= \frac1{2^{\frac{p+q}2}} \Derc{\cC}_X \big((\nu \!\bcdot\! \f^{\zero}) \!\otimes\! (\overline{\mu} \!\bcdot\! \psi^{\zero})\big) \\
&=\frac1{2^{\frac{p+q}2}} \big((\Derc{\cC}_X\nu \!\bcdot\! \f^{\zero}) \!\otimes\! (\overline{\mu} \!\bcdot\! \psi^{\zero}) + (\nu \!\bcdot\! \f^{\zero})\!\otimes\! (\Derc{\cC}_X\overline{\mu} \!\bcdot\! \psi^{\zero})\big) \\
&=\big(\b^{{}^\downarrow p} \!\otimes\! \a^{{}^\uparrow q}\big)\big((\Derc{\cC}_X\nu \!\otimes\! \f^{\zero}) \!\otimes\! (\overline{\mu} \!\otimes\! \psi^{\zero}) + (\nu \!\otimes\! \f^{\zero})\!\otimes\! (\Derc{\cC}_X\overline{\mu} \!\otimes\! \psi^{\zero})\big) \\
&=\big(\b^{{}^\downarrow p} \!\otimes\! \a^{{}^\uparrow q}\big)\Big((\jmath^{\, p,q})^{-1}\big((\Derc{\cC}_X\nu \wedge \overline{\mu}) \!\otimes\! \x^{\zero} + (\nu \wedge \Derc{\cC}_X\overline{\mu}) \!\otimes\! \x^{\zero}\big)\Big) \\
&=\vars^{p,q}\big(\Derc{\cC}_X\h\big)
\end{align*} and this completes the proof. \end{proof}

The components $\vars^{p,q}$ of the isometry $\vars$ defined in \eqref{ropq} play an important role in the following discussion. It is therefore convenient to introduce the notation $$\widehat{\bcdot} \big|_{\L^{p,q}(T^*M) \otimes \eV^{0,0}M} \= \Big(\!\big(\!\bcdot\! |_{\L^{p,0}(T^*M) \otimes \eS^{{}^\downarrow n}M}\big) \!\otimes\! \big(\!\bcdot\! |_{\L^{0,q}(T^*M) \otimes \eS^{{}^\uparrow 0}M}\big)\!\Big) \circ (\jmath^{\, p,q})^{-1}\, ,$$ which allows to shortly indicate the map $\vars^{p,q}$ as \beq \vars^{p,q}(\h) = \frac1{2^{\frac{p+q}2}} \h \widehat{\bcdot} \xi^{\zero} \quad \text{ for every } \h \in \W^{p,q}(M)\,\, . \label{notwedge} \eeq

\subsection{The algebraic structure of $\eV$-spinors}

Let $(M,g,J)$ be an Hermitian $2n$-manifold and $\Cl^{\,\bC}\!M$ its complex Clifford bundle. Since $\Cl^{\, \bC}_{2n} \simeq \gg\gl(\eS_{2n}) = \cM_{2^n}(\bC)$, there exists a $\bC$-linear isomorphism \beq \c: \eV M \ra \Cl^{\,\bC}\!M \simeq \gg\gl(\eS^{{}^\uparrow}\!M) \, ,\label{chi} \eeq which maps each decomposable $\eV$-spinor $\f {\otimes} \psi \in \eV M$ into the unique element $w\= \c(\f {\otimes} \psi) \in \Cl^{\, \bC}\!M$ defined by the condition $$w \!\bcdot\! \psi' = \f(\psi')\psi \quad \text{ for every } \psi' \in \eS^{{}^\uparrow}\!M \,\, .$$ This gives rise to the following $\bC$-linear isomorphism between $\L^{\bcdot}(T^{*\bC}M)$ and $\Cl^{\,\bC}\!M$: \beq \c \circ \vars : \L^{\bcdot}(T^{*\bC}M) \ra \Cl^{\,\bC}\!M \, . \label{isoVM} \eeq The reader should nonetheless be aware that {\it such map is in general different from the canonical isomorphism $\Cl^{\,\bC}\!M \simeq \L^{\bcdot}(T^{\bC}M) \overset{g}{\simeq} \L^{\bcdot}(T^{*\bC}M)$ described in \eqref{isoGrCl}}. Furthermore, the isomorphism \eqref{isoVM} induces a bigradation on $\Cl^{\,\bC}\!M$, defined by \beq \Cl^{p,q}\!M \= \c(\eV^{p,q}M) = (\c \!\circ\! \vars)\big(\L^{p,q}(T^*M)\big) \,\, , \label{Clpq} \eeq which is different from the one considered by Michelsohn in \cite{Mi}, \S 2.B. \smallskip

We can finally define a structure of bundle of algebras on $\eV M$ setting $$\x_1 \cdot \x_2 \= \c^{-1}\big(\c(\x_1) \cdot \c(\x_2)\big) \quad \text{ for every } \x_1 , \x_2 \in \eV M \,\, .$$ Note that if $\x_1= \f_1 \!\otimes\! \psi_1$ and $\x_2= \f_2 \!\otimes\! \psi_2$ are homogeneous decomposable $\eV$-spinors, with $\f_j = \nu_j \!\bcdot\! \f^{\zero} \in \eS^{{}^\downarrow n-p_j}M$ and $\psi_j = \overline{\mu}_j \!\bcdot\! \psi^{\zero} \in \eS^{{}^\uparrow q_j}M$, from \eqref{isodelta} and \eqref{ropq} it follows that such product is nothing but $$\x_1 \cdot \x_2 = (\f_1 \!\otimes\! \psi_1) \cdot (\f_2 \!\otimes\! \psi_2) = \f_1(\psi_2) (\f_2 \!\otimes\! \psi_1) = 2^{\frac{p_2+q_1}2}g(\nu_1, \overline{\mu}_2) \,\vars(\nu_2 \wedge \overline{\mu}_1) \ .$$

\subsection{Partial Chern-Dirac operators on $\eV$-spinors}

Let $(M,g,J)$ be an Hermitian $2n$-manifold. The main result of this section consists in the proof that the partial Chern-Dirac operators on $\eV M$ correspond to the standard operators $\p$, $\p^*$ and $\bar{\p}^*$, $\bar{\p}$ on differential forms. \smallskip

Let us indicate with $\Diracde\!'{}^{{}^{(L)}}$, $\Diracde\!''{}^{{}^{(L)}}$ and $\Diracde\!'{}^{{}^{(R)}}$, $\Diracde\!''{}^{{}^{(R)}}$ the partial Chern-Dirac operators on sections of $\eV M$ determined in \eqref{pCD} using the Clifford multiplications $c= \bcdot_{\!\!{}_L}$ and $c= \bcdot_{\!\!{}_R}$, respectively. In full analogy with Theorem 4.1 in \cite{Mi}, we prove the following

\begin{theorem} The operators $\Diracde\!'{}^{{}^{(L)}}\!\!$ and $\Diracde\!''{}^{{}^{(L)}}\!\!$ (resp. $\Diracde\!'{}^{{}^{(R)}}\!\!$ and $\Diracde\!''{}^{{}^{(R)}}\!$) are formal adjoint one of the another and their squares satisfy \beq \Big(\Diracde\!'{}^{{}^{(L)}}\Big)^2 = 0 = \Big(\Diracde\!''{}^{{}^{(L)}}\Big)^2 \, , \,\,\quad\,\, \Big(\Diracde\!'{}^{{}^{(R)}}\Big)^2 = 0 = \Big(\Diracde\!''{}^{{}^{(R)}}\Big)^2 \,\, .\label{zerosquare} \eeq Moreover, the cochain complexes \beq \begin{split} \gV^{0,q}(M) \overset{\Diracde\!'{}^{{}^{(L)}}}{\xrightarrow{\hspace*{30pt}}} \gV^{1,q}(M) \overset{\Diracde\!'{}^{{}^{(L)}}}{\xrightarrow{\hspace*{30pt}}} \,\cdots\, \overset{\Diracde\!'{}^{{}^{(L)}}}{\xrightarrow{\hspace*{30pt}}} \gV^{n,q}(M) \, , \\ \gV^{p,0}(M) \overset{\Diracde\!''{}^{{}^{(R)}}}{\xrightarrow{\hspace*{30pt}}} \gV^{p,1}(M) \overset{\Diracde\!''{}^{{}^{(R)}}}{\xrightarrow{\hspace*{30pt}}} \,\cdots\, \overset{\Diracde\!''{}^{{}^{(R)}}}{\xrightarrow{\hspace*{30pt}}} \gV^{p,n}(M) \, \phantom{,} \end{split} \label{complexes} \eeq are elliptic. \end{theorem}

\begin{proof} The first claim follows directly from Proposition \ref{adjoint}, while \eqref{zerosquare} follows from \eqref{Rsquare} and properties of the curvature of Chern connection. In order to prove that $\Diracde\!'{}^{{}^{(L)}}\!\big(\gV^{p,q}(M)\big) \subset \gV^{p+1,q}(M)$, consider a unitary frame field defined in an open subset $\cU \subset M$ and the associated normalized complex frame $(\e_j, \overline{\e}_j)$ defined in $\eqref{vettu}$. One can directly check that \begin{itemize}
\item[i)] $\displaystyle \sum_j \w \!\cdot\! \overline{\e}_j = \sum_j (\overline{\e}_j \!\cdot\! \w -2i\, \overline{\e}_j)$ , \\
\item[ii)] $\displaystyle \sum_{r<s} \w \!\cdot\! \overline{\e}_r \!\cdot\! \overline{\e}_s \!\cdot\! T_{rs} = \sum_{r<s} (\overline{\e}_r \!\cdot\! \overline{\e}_s \!\cdot\! T_{rs} \!\cdot\! \w - 2i\, \overline{\e}_r \!\cdot\! \overline{\e}_s \!\cdot\! T_{rs})$ .
\end{itemize}
Fix now a homogeneous section $\x \in \gV^{p,q}(M)$. Since $\Derc{\cC}_{\color{white} X}\!\!\w=0$, from $(i)$ and $(ii)$ we obtain \begin{align*}
\w \bcdot_{\!\!{}_L} \Diracde\!'{}^{{}^{(L)}}\x &= \sum_j (\w \!\cdot\! \overline{\e}_j)\bcdot_{\!\!{}_L} \Derc{\cC}_{\e_j}\x - \frac12\sum_{r<s}(\w \!\cdot\! \overline{\e}_r \!\cdot\! \overline{\e}_s \!\cdot\! T_{rs}) \bcdot_{\!\!{}_L} \x \\
&=(n-2p-2)i \sum_j \overline{\e}_j \bcdot_{\!\!{}_L} \Derc{\cC}_{\e_j}\x -\frac12 (n-2p-2)i \sum_{r<s} (\overline{\e}_r \!\cdot\! \overline{\e}_s \!\cdot\! T_{rs}) \bcdot_{\!\!{}_L} \x \\
&=(n-2(p+1))i\,\Diracde\!'{}^{{}^{(L)}}\x \, ,
\end{align*} and then it is immediate to check that $\w \bcdot_{\!\!{}_R} \Diracde\!'{}^{{}^{(L)}}\x = (-1)^{p+1}(2q-n)i\,\Diracde\!'{}^{{}^{(L)}}\x$. One can show in a similar way the inclusion $\Diracde\!''{}^{{}^{(R)}}\!\big(\gV^{p,q}(M)\big) \subset \gV^{p,q+1}(M)$. It remains to prove that \eqref{complexes} are elliptic. To see this, consider a real covector $\l \in TM$ and observe that the principal symbols of the four partial Chern-Dirac operators are $$\s\Big(\Diracde\!'{}^{{}^{(L)}}\Big)(\l) = \big(\l^{\sharp}\big)^{10}\bcdot_{\!\!{}_L} \,\, , \quad \s\Big(\Diracde\!''{}^{{}^{(L)}}\Big)(\l) = \big(\l^{\sharp}\big)^{01}\bcdot_{\!\!{}_L} \,\, ,$$ $$\s\Big(\Diracde\!''{}^{{}^{(R)}}\Big)(\l) = \big(\l^{\sharp}\big)^{01}\bcdot_{\!\!{}_R} \,\, , \quad \s\Big(\Diracde\!'{}^{{}^{(R)}}\Big)(\l) = \big(\l^{\sharp}\big)^{10}\bcdot_{\!\!{}_R} \,\, .$$ Ellipticity follows now from the relation $$\big(\l^{\sharp}\big)^{10}\!\cdot\!\big(\l^{\sharp}\big)^{01}+\big(\l^{\sharp}\big)^{01}\!\cdot\!\big(\l^{\sharp}\big)^{10} = -g(\l,\l)$$ and this completes the proof. \end{proof}

The next theorem establishes a crucial relation between the partial Chern-Dirac operators on $\eV$-spinors and the standard operators $\p$, $\p^*$ and $\bar{\p}^*$, $\bar{\p}$ on differential forms. One can immediately realize that this relation is much simpler than the analogous result on the operators $\eD$, $\overline{\eD}$,  considered by Michelsohn for Clifford bundles on K\"ahler manifolds (see \cite{Mi}, Proposition 5.1).

\begin{theorem} Let $\vars: \L^{\bcdot}(T^{*\bC}M) \ra \eV M$ be the isometry defined in Proposition \ref{propVM}. Then the partial Chern-Dirac operators verify \begin{align*}
\vars^{-1} \!\circ\! \Diracde\!'{}^{{}^{(L)}} \!\circ\! \vars = \sqrt2 \partial \ &, \quad \vars^{-1} \!\circ\! \Diracde\!''{}^{{}^{(L)}} \!\circ\! \vars = \sqrt2 \partial^* \ , \\
\vars^{-1} \!\circ\! \Diracde\!'{}^{{}^{(R)}} \!\circ\! \vars = \sqrt2 \bar{\partial}^* \ &, \quad \vars^{-1} \!\circ\! \Diracde\!''{}^{{}^{(R)}} \!\circ\! \vars = \sqrt2 \bar{\partial} \,\, . \end{align*} \label{isopartCD} \end{theorem}
For the proof, we need a preparatory lemma.

\begin{lemma} Let $x \!\in\! M$ and $\h \in \W^{p,q}(M)$. Then \begin{align*}
&\partial \h|_x = \sum_j \Big(\e^j \wedge \Derc{\cC}_{\e_j}\h +\e^j(T) \!\wedge\! \big(\e_j \lrcorner \h\big)\Big) \,\, , \\
&\bar{\partial}\h|_x = \sum_j \Big(\overline{\e}^j \wedge \Derc{\cC}_{\overline{\e}_j}\h +\overline{\e}^j(T) \!\wedge\! \big(\overline{\e}_j \lrcorner \h\big)\Big) \,\, , \\
&\partial^*\h|_x = -\sum_j \e_j \lrcorner \Derc{\cC}_{\overline{\e}_j}\h - \sum_{r<s} \e_r \lrcorner \e_s \lrcorner \big((T_{\bar{r}\bar{s}})^\flat \!\wedge\!\h\big) \,\, , \\
&\bar{\partial}^*\h|_x = -\sum_j \overline{\e}_j \lrcorner \Derc{\cC}_{\e_j}\h - \sum_{r<s} \overline{\e}_r \lrcorner \overline{\e}_s \lrcorner \big((T_{rs})^\flat \!\wedge\!\h\big) \,\, ,
\end{align*} where $(\e_j, \overline{\e}_j)$ is the usual normalized complex frame \eqref{vettu} determined by a unitary basis for $T_xM$. \label{lemmade} \end{lemma}
\begin{proof} Since $S=\Derc{\cC}_{\color{white} X}-\Derc{LC}_{\color{white} X}$, the contorsion tensor acts in a natural way on forms. Actually, it can be proved by induction that its action is such that $$S_X\big(\W^{p,q}(M)\big) \subset \W^{p+1,q-1}(M) \oplus \W^{p,q}(M) \oplus \W^{p-1,q+1}(M) \ , \quad 0 \leq p,q \leq n$$ for every $X \in \gX(M)$. From standard properties of Levi-Civita connection and \eqref{contS}, given an unitary basis $(e_j) \subset T_xM$, we get $$d\h|_x = \sum_{s=1}^{2n} \Big(e^s \wedge \Derc{LC}_{e_s}\h\Big) = \sum_j \Big(\e^j \wedge \Derc{\cC}_{\e_j}\h+\overline{\e}^j \wedge \Derc{\cC}_{\overline{\e}_j}\h-\e^j \wedge S_{\e_j}\h-\overline{\e}^j \wedge S_{\overline{\e}_j}\h\Big)$$ and so \begin{align*}
&\partial\h|_x = \big(d\h|_x\big)^{p+1,q} = \sum_j \Big(\e^j \wedge \Derc{\cC}_{\e_j}\h - \e^j \!\wedge\! \big(S_{\e_j}\h\big)^{p,q}-\overline{\e}^j \!\wedge\! \big(S_{\overline{\e}_j}\h\big)^{p+1,q-1}\Big) \ , \\
&\bar{\partial}\h|_x = \big(d\h|_x\big)^{p,q+1} = \sum_j \Big(\overline{\e}^j \wedge \Derc{\cC}_{\overline{\e}_j}\h - \e^j \!\wedge\! \big(S_{\e_j}\h\big)^{p-1,q+1}-\overline{\e}^j \!\wedge\! \big(S_{\overline{\e}_j}\h\big)^{p,q}\Big) \ .
\end{align*} Furthermore, one can directly check that \begin{align*}
&\sum_j \Big(- \e^j \!\wedge\! \big(S_{\e_j}\h\big)^{p,q}-\overline{\e}^j \!\wedge\! \big(S_{\overline{\e}_j}\h\big)^{p+1,q-1}\Big)= \sum_j \e^j(T) \!\wedge\! \big(\e_j \lrcorner \h\big) \, ,\\
&\sum_j \Big(- \e^j \!\wedge\! \big(S_{\e_j}\h\big)^{p-1,q+1}-\overline{\e}^j \!\wedge\! \big(S_{\overline{\e}_j}\h\big)^{p,q}\Big)= \sum_j \overline{\e}^j(T) \!\wedge\! \big(\overline{\e}_j \lrcorner \h\big) \, .
\end{align*} The remaining identities can be proved in a similar way. \end{proof}

We may now proceed with the proof of Theorem \ref{isopartCD}.

\begin{proof}[Proof of Theorem \ref{isopartCD}] Consider a unitary frame field $(e_j): \cU \subset M \ra \U_{\!{}_{g,J}\!}(M)$, the normalized complex vectors $\e_j$, $\overline{\e}_j$ defined in \eqref{vettu} and a form $\h = \nu \wedge \overline{\mu} \in \W^{p,q}(M)$. Then, using notation \eqref{notwedge}, from \eqref{prod1form} and Lemma \ref{lemmade} it follows that \begin{align*}
\Diracde\!'{}^{{}^{(L)}}\!(\h \widehat{\bcdot} \xi^{\zero}) \!&= \sum_j \overline{\e}_j \!\bcdot_{\!\!{}_L} \Derc{\cC}_{\e_j}(\h \widehat{\bcdot} \xi^{\zero}) - \frac12\sum_{r<s} (\overline{\e}_r \!\cdot\! \overline{\e}_s \!\cdot\! T_{rs}) \!\bcdot_{\!\!{}_L}(\h \widehat{\bcdot} \xi^{\zero}) \\
\phantom{\Diracde\!'{}^{{}^{(L)}}\!(\h \widehat{\bcdot} \xi^{\zero}) \!}&=\!\!\sum_j \!\Big(\overline{\e}_j \!\bcdot_{\!\!{}_{L}}\! \big((\Derc{\cC}_{\e_j}\nu \!\wedge\! \overline{\mu}) \widehat{\bcdot} \xi^{\zero}\big) {+}\overline{\e}_j \!\bcdot_{\!\!{}_{L}} \!\big((\nu \!\wedge\! \Derc{\cC}_{\e_j}\overline{\mu}) \widehat{\bcdot} \xi^{\zero}\big)\!\Big){-} \frac12\!\sum_{r<s} (\overline{\e}_r \!\cdot\! \overline{\e}_s \!\cdot\! T_{rs}) \!\bcdot_{\!\!{}_L}\!(\h \widehat{\bcdot} \xi^{\zero}) \\
&=\sum_j \Big((\e^j \!\cdot\! \Derc{\cC}_{\e_j}\nu \!\bcdot\! \f^{\zero}) \!\otimes\! (\overline{\mu} \!\bcdot\! \psi^{\zero}){+}(\e^j \!\cdot\! \nu \!\bcdot\! \f^{\zero}) \!\otimes\! (\Derc{\cC}_{\e_j}\overline{\mu} \!\bcdot\! \psi^{\zero})\Big) - \\
\phantom{\Diracde\!'{}^{{}^{(L)}}\!(\h \widehat{\bcdot} \xi^{\zero}) \!}&\phantom{aaaaaaaaaaaaaaaaaaaaaaaaaaaaai}-\frac12\sum_{\substack{r<s \\ m}}\,T_{rs\overline{m}}(\overline{\e}_r \!\cdot\! \overline{\e}_s \!\cdot\! \e_m \!\cdot\! \nu \!\bcdot\! \f^{\zero}) \!\otimes\! (\overline{\mu} \!\bcdot\! \psi^{\zero}) \\
\phantom{\Diracde\!'{}^{{}^{(L)}}\!(\h \widehat{\bcdot} \xi^{\zero}) \!}&=\sum_j \Big(\big((\e^j \!\wedge\! \Derc{\cC}_{\e_j}\nu) \!\bcdot\! \f^{\zero}\big) \!\otimes\! (\overline{\mu} \!\bcdot\! \psi^{\zero}) {+}\big((\e^j \!\wedge\! \nu) \!\bcdot\! \f^{\zero}\big) \!\otimes\! (\Derc{\cC}_{\e_j}\overline{\mu} \!\bcdot\! \psi^{\zero})\Big)+\\
\phantom{\Diracde\!'{}^{{}^{(L)}}\!(\h \widehat{\bcdot} \xi^{\zero}) \!}&\phantom{aaaaaaaaaaaaaaaaaaaaaaaaaa}+\sum_{\substack{r<s \\ m}}\,T_{rs\overline{m}}\,\big((\e^r \!\wedge\! \e^s \!\wedge\! (\e_m \lrcorner \nu)) \!\bcdot\! \f^{\zero}\big) \!\otimes\! (\overline{\mu} \!\bcdot\! \psi^{\zero}) \\
\phantom{\Diracde\!'{}^{{}^{(L)}}\!(\h \widehat{\bcdot} \xi^{\zero}) \!}&=\sum_j (\e^j \!\wedge\! \Derc{\cC}_{\e_j}\h)\widehat{\bcdot} \xi^{\zero} + \sum_{\substack{r<s \\ m}}\,T_{rs\overline{m}}\,(\e^r \!\wedge\! \e^s \!\wedge\! (\e_m \lrcorner \h))\widehat{\bcdot} \xi^{\zero}
\end{align*}

\noindent $\displaystyle \phantom{\Diracde\!'{}^{{}^{(L)}}\!(\h \widehat{\bcdot} \xi^{\zero}) \!}=\sum_j \big(\e^j \!\wedge\! \Derc{\cC}_{\e_j}\h + \e^j(T) \!\wedge\! (\e_j \lrcorner \h)\big)\widehat{\bcdot} \xi^{\zero}$ \\
$\displaystyle \phantom{\Diracde\!'{}^{{}^{(L)}}\!(\h \widehat{\bcdot} \xi^{\zero}) \!}=\partial\h \widehat{\bcdot} \xi^{\zero}$ . \\
So we get: $$\big(\vars^{-1} \!\circ\! \Diracde\!'{}^{{}^{(L)}} \!\circ \vars\big)(\h) =\frac1{2^{\frac{p+q}2}}\vars^{-1}\big(\partial\h \widehat{\bcdot} \xi^{\zero}\big) =\frac{2^{\frac{p+q+1}2}}{2^{\frac{p+q}2}}\partial \h =\sqrt2 \partial \h \ .$$ Proceeding in a similar way we obtain \begin{align*}
\Diracde\!''{}^{{}^{(L)}}\!(\h \widehat{\bcdot} \xi^{\zero}) &= \sum_j \e_j \!\bcdot_{\!\!{}_L} \Derc{\cC}_{\overline{\e}_j}(\h \widehat{\bcdot} \xi^{\zero}) - \frac12\sum_{r<s} (\e_r \!\cdot\! \e_s \!\cdot\! T_{\bar{r}\bar{s}}) \!\bcdot_{\!\!{}_L}(\h \widehat{\bcdot} \xi^{\zero}) \\
&=\!\!\sum_{j=1}^n \!\!\Big(\!\e_j \!\bcdot_{\!\!{}_L} \!\big((\Derc{\cC}_{\overline{\e}_j}\nu \!\wedge\! \overline{\mu}) \widehat{\bcdot} \xi^{\zero}\big) {+} \e_j  \!\bcdot_{\!\!{}_L}\! \big((\nu \!\wedge\! \Derc{\cC}_{\overline{\e}_j}\overline{\mu}) \widehat{\bcdot} \xi^{\zero}\big)\!\Big) {-} \frac12\!\sum_{r<s}\! (\e_r \!\cdot\! \e_s \!\cdot\! T_{\bar{r}\bar{s}}) \!\bcdot_{\!\!{}_L}\!(\h \widehat{\bcdot} \xi^{\zero}) \\
\phantom{\Diracde\!''{}^{{}^{(L)}}\!(\h \widehat{\bcdot} \xi^{\zero})} &=\sum_{j=1}^n \Big((\overline{\e}^j \!\cdot\! \Derc{\cC}_{\overline{\e}_j}\nu \!\bcdot\! \f^{\zero}) \!\otimes\! (\overline{\mu} \!\bcdot\! \psi^{\zero}){+}(\overline{\e}^j \!\cdot\! \nu \!\bcdot\! \f^{\zero}) \!\otimes\! (\Derc{\cC}_{\overline{\e}_j}\overline{\mu} \!\bcdot\! \psi^{\zero})\Big) -\\
&\phantom{aaaaaaaaaaaaaaaaaaaaaaaaaaaaai} {-} \frac12 \sum_{\substack{r<s \\ m}} T_{\bar{r}\bar{s}m} \,(\e_r \!\cdot\! \e_s \!\cdot\! \overline{\e}_m \!\cdot\! \nu \!\bcdot\! \f^{\zero})\!\otimes\! (\overline{\mu} \!\bcdot\! \psi^{\zero}) \\
&=\sum_{j=1}^n\Big(\big({-}2(\e_j \lrcorner \Derc{\cC}_{\overline{\e}_j}\nu) \!\bcdot\! \f^{\zero}\big) \!\otimes\! (\overline{\mu} \!\bcdot\! \psi^{\zero}) {+}\big({-}2(\e_j \lrcorner \nu) \!\bcdot\! \f^{\zero}\big) \!\otimes\! (\Derc{\cC}_{\overline{\e}_j}\overline{\mu} \!\bcdot\! \psi^{\zero})\Big) +\\
&\phantom{aaaaaaaaaaaaaaaaaaaaaaai}+\!\sum_{\substack{r<s \\ m}}T_{\bar{r}\bar{s}m} \,\big({-}2(\e_r \lrcorner \e_s \lrcorner (\e^m \!\wedge\! \nu)) \!\bcdot\! \f^{\zero}\big)\!\otimes\! (\overline{\mu} \!\bcdot\! \psi^{\zero}) \\
&=\sum_j \big({-}2(\e_j \lrcorner \Derc{\cC}_{\overline{\e}_j}\h)\widehat{\bcdot} \xi^{\zero}\big)+ \sum_{\substack{r<s \\ m}} T_{\bar{r}\bar{s}m}\,\big({-}2(\e_r \lrcorner \e_s \lrcorner (\e^m \!\wedge\! \h))\widehat{\bcdot} \xi^{\zero}\big) \\
&=2\Big({-}\sum_j \e_j \lrcorner \Derc{\cC}_{\overline{\e}_j}\h -\!\! \sum_{r<s} \e_r \lrcorner \e_s \lrcorner \big((T_{\bar{r}\bar{s}})^\flat \!\wedge\!\h\big)\Big)\widehat{\bcdot} \xi^{\zero} \\
&=2\partial^*\h \widehat{\bcdot} \xi^{\zero}
\end{align*} and thus $$\big(\vars^{-1} \!\circ\! \Diracde\!''{}^{{}^{(L)}} \!\circ \vars\big)(\h) =\frac1{2^{\frac{p+q}2}}\vars^{-1}\big(2\partial^*\h \widehat{\bcdot} \xi^{\zero}\big) =2\frac{2^{\frac{p+q-1}2}}{2^{\frac{p+q}2}}\partial^* \h =\sqrt2 \partial^* \h \ .$$ The remaining two cases are perfectly analogous. \end{proof}

\subsection{Harmonic $\eV$-spinors}

Assume now that $(M,g,J)$ is a compact Hermitian $2n$-manifold and consider the Chern-Dirac operators on $\eV M$ determined by \eqref{CD} using the Clifford multiplications $c= \bcdot_{\!\!{}_L}$ and $c= \bcdot_{\!\!{}_R}$, i.e. \beq \Dirac{}^{\!\!\!{}^{(L)}}\! = \Diracde\!'{}^{{}^{(L)}} + \Diracde\!''{}^{{}^{(L)}} \, , \quad\, \Dirac{}^{\!\!\!{}^{(R)}}\! = \Diracde\!'{}^{{}^{(R)}} + \Diracde\!''{}^{{}^{(R)}} \ . \eeq We call them {\it left Chern-Dirac operator} and {\it right Chern-Dirac operator on $\eV$-spinors}, respectively. As we pointed out in \S4, they are both first order elliptic operators and, since $M$ is compact, they are self-adjoint too. We call {\it total-harmonic $\eV$-spinors} the sections of $\eV M$ which are in the kernel of $\Dirac{}^{\!\!\!{}^{(L)}}\!\! {+} \Dirac{}^{\!\!\!{}^{(R)}}$ and {\it right-}(resp. {\it left-}){\it harmonic $\eV$-spinors} the sections of $\eV M$ which are in the kernel of $\Dirac{}^{\!\!\!{}^{(R)}}$ (resp. $\Dirac{}^{\!\!\!{}^{(L)}}$). From Theorem \ref{isopartCD} we obtain the following isomorphism between spaces of harmonic $\eV$-spinors and cohomology groups. 

\begin{theo} Let $(M,g,J)$ be a compact Hermitian $2n$-manifold. Then $$\ker{\!\Big(\!\Dirac{}^{\!\!\!{}^{(L)}}\!\! {+} \Dirac{}^{\!\!\!{}^{(R)}}\!\Big)} \simeq \bigoplus_{k=0}^{2n} H^k_d(M;\bC) \ , \quad \ker{\Dirac{}^{\!\!\!{}^{(R)}}\!} \simeq \bigoplus_{p,q=0}^n H^{p,q}_{\bar{\partial}}(M) \, ,$$ where $H^k_d(M;\bC)$ and $H^{p,q}_{\bar{\partial}}(M)$ are the usual De Rham and Dolbeault cohomology groups of $M$. \label{teoDdR} \end{theo}
\begin{proof} Since for every $0 \leq k \leq 2n$ the spaces $d\big(\W^{k-1}(M;\bC)\big)$ and $d^*\big(\W^{k+1}(M;\bC)\big)$ are orthogonal, from Theorem \ref{isopartCD} and standard Hodge theory it follows that \begin{align*}
\ker{\!\Big(\!\Dirac{}^{\!\!\!{}^{(L)}}\!\! {+} \Dirac{}^{\!\!\!{}^{(R)}}\!\Big)} &\overset{\vars}{\simeq} \big\{\h \in \W^{\bcdot}(M;\bC) : d\h + d^*\h = 0 \big\} \\
\phantom{\ker{\!\Big(\!\Dirac{}^{\!\!\!{}^{(L)}}\!\! {+} \Dirac{}^{\!\!\!{}^{(R)}}\!\Big)} }&= \bigoplus_{k=0}^{2n} \big\{\h \in \W^k(M;\bC) : d\h = d^*\h= 0 \big\} \\
& \simeq \bigoplus_{k=0}^{2n} H^k_d(M;\bC) \ .
\end{align*} The proof of the second isomorphism is similar. \end{proof}

\begin{rem} We recall that the canonical spinor bundle of $M$ is naturally included in $\eV M$ (see  $\S 4.2$). Due to this, Theorem \ref{teoDdR} can be considered as a generalization of the well-known isomorphism between harmonic spinors and cohomology classes $H^{0,\cdot}_{\bar{\partial}}(M)$ on compact K\"ahler manifolds (\cite{Hi}, Thm. 2.1). \end{rem}

We conclude this section showing how the partial Chern-Dirac operators on $\eV$-spinors can be used to provide a spinorial characterization for the {\it Bott-Chern} and {\it Aeppli cohomologies} \beq H^{\bcdot,\bcdot}_{{}_{BC}}(M) \= \frac{\ker\p \cap \ker\bar{\p}}{\Im \p\bar{\p}} \,\, , \quad\, H^{\bcdot,\bcdot}_{{}_A}(M) \= \frac{\ker \p\bar{\p}}{\Im \p + \Im\bar{\p}} \,\, . \eeq We recall that both such cohomologies coincide with the usual Dolbeault cohomology in the K\"ahler case, and they are an important tool in studies on non-K\"ahler Hermitian manifolds. For an introduction to them, see e.g. \cite{AT1, AT2}. \smallskip

There is an Hodge Theory for these two cohomologies. In fact, they satisfy \beq H^{\bcdot,\bcdot}_{{}_{BC}}(M) \simeq \ker \D_{{}_{BC}} \,\, , \quad\, H^{\bcdot,\bcdot}_{{}_A}(M) \simeq \ker \D_{{}_A}\,\, , \label{HodgeBCA}\eeq where $\D_{{}_{BC}}$ and $\D_{{}_A}$ are the $4^{th}$ order elliptic self-adjoint operators defined by $$\begin{gathered} \D_{{}_{BC}} \= \big(\p\bar{\p}\big)\big(\p\bar{\p}\big)^*+\big(\p\bar{\p}\big)^*\big(\p\bar{\p}\big)+\big(\bar{\p}^*\p\big)\big(\bar{\p}^*\p\big)^*+\big(\bar{\p}^*\p\big)^*\big(\bar{\p}^*\p\big)+\bar{\p}^*\bar{\p}+\p^*\p \,\, , \\
\D_{{}_A} \= \big(\bar{\p}\p^*\big)\big(\bar{\p}\p^*\big)^*+\big(\bar{\p}\p^*\big)^*\big(\bar{\p}\p^*\big)+\big(\p\bar{\p}\big)\big(\p\bar{\p}\big)^*+\big(\p\bar{\p}\big)^*\big(\p\bar{\p}\big)+\bar{\p}\bar{\p}^*+\p\p^* \,\, . \end{gathered}$$ Due to the fact that for every complex form $\h$ one has \beq \begin{gathered} \D_{{}_{BC}}\h =0 \quad \text{ if and only if } \quad \p\h=\bar{\p}\h=\bar{\p}^*\p^*\h=0 \,\, , \\
\D_{{}_A}\h =0 \quad \text{ if and only if } \quad \p^*\h=\bar{\p}^*\h=\p\bar{\p}\h=0 \,\, , \end{gathered} \label{HodgeBCA2} \eeq it is natural to consider the operators \begin{gather*}
\Dirac{}_{\!\!\!\!{}_{BC}} , \ \Dirac{}_{\!\!\!\!{}_A} : \gV(M) \ra \gV(M) \ ,\\
\Dirac{}_{\!\!\!\!{}_{BC}} \= \Diracde\!'{}^{{}^{(L)}} \!+ \Diracde\!''{}^{{}^{(R)}} \!+ \Diracde\!'{}^{{}^{(R)}} \!\circ \Diracde\!''{}^{{}^{(L)}} \, , \quad \,\,\,
\Dirac{}_{\!\!\!\!{}_A} \= \Diracde\!''{}^{{}^{(L)}} \!+ \Diracde\!'{}^{{}^{(R)}} \!+ \Diracde\!'{}^{{}^{(L)}} \!\circ \Diracde\!''{}^{{}^{(R)}} \ ,
\end{gather*} which we call {\it Bott-Chern-Dirac operator} and {\it Aeppli-Dirac operator on $\eV$-spinors}. From \eqref{HodgeBCA}, \eqref{HodgeBCA2} and Theorem \ref{isopartCD}, it follows directly the following

\begin{prop} Let $(M,g,J)$ be a compact Hermitian $2n$-manifold. Then, for every $0 \leq p,q \leq n$, the kernels of $\Dirac{}_{\!\!\!\!{}_{BC}}$ and $\Dirac{}_{\!\!\!\!{}_A}$ verify $$H^{p,q}_{{}_{BC}}(M) \simeq \ker{\Dirac{}_{\!\!\!\!{}_{BC}}} \cap \gV^{p,q}(M) \,\, , \quad H^{p,q}_{{}_A}(M) \simeq \ker{\Dirac{}_{\!\!\!\!{}_A}} \cap \gV^{p,q}(M)$$ and so there exist injective homomorphisms $$\bigoplus_{p,q=0}^n H^{p,q}_{{}_{BC}}(M) \hookrightarrow \ker{\Dirac{}_{\!\!\!\!{}_{BC}}} \ , \quad \bigoplus_{p,q=0}^n H^{p,q}_{{}_A}(M) \hookrightarrow \ker{\Dirac{}_{\!\!\!\!{}_A}} \ .$$ \end{prop}

\section{Twisted cohomology of Hermitian manifolds and $\eV$-spinors} \setcounter{equation} 0

\subsection{Twisted $\eV$-spinor bundle with respect to an Hermitian bundle}

Let $(M,g,J)$ be an Hermitian $2n$-manifold and let $E$ be a Chern-Dirac bundle over $M$ with Clifford multiplication $c: \Cl^{\, \bC}M \ra \gg\gl(E)$, Hermitian metric $h$ and covariant derivative $D$. Let also $\big(W,h^{{}^W}\big)$ be an Hermitian bundle over $M$, endowed with a metric covariant derivative $\Derc{W}_{\color{white} X}$. We can trivially extend $c$ to the tensor product bundle $E {\otimes} W$ by \beq c(w)\big(\s {\otimes} s\big) \= \big(c(w)\s\big) {\otimes} s \label{Clifmulttw} \eeq and define \beq \widetilde{h} \= h {\otimes} h^{{}^W} \,\, , \quad \widetilde{D}\= D {\otimes} \Id_{{}_W} + \Id_{{}_E} {\otimes} \Derc{W}_{\color{white} X} \,\, . \label{CDstruct} \eeq With a straightforward computation, one can directly check the following

\begin{prop} The tensor product bundle $\big(E {\otimes} W,\widetilde{h}, \widetilde{D}\big)$ is a Chern-Dirac bundle with respect to the Clifford multiplication \eqref{Clifmulttw}. \end{prop}

We focus now on the case in which $E= \eV M$. We call $\eV M {\otimes} W$ the {\it $W$-twisted $\eV$-spinor bundle}. It is naturally endowed with the four {\it twisted partial Chern-Dirac operators} \beq \widetilde{\Diracde\!'{}^{{}^{(L)}}} \,\, , \quad \widetilde{\Diracde\!''{}^{{}^{(L)}}} \,\, , \quad \widetilde{\Diracde\!'{}^{{}^{(R)}}} \,\, , \quad \widetilde{\Diracde\!''{}^{{}^{(R)}}} \,\, , \label{twpCD} \eeq defined in \eqref{pCD}. On the other hand, we recall that the covariant derivative $\Derc{W}_{\color{white} X}\!\!$ defines the exterior derivative $d^{{}^W}\!\!: \W^k(M;W) \ra \W^{k+1}(M;W)$ on $W$-valued differential forms \beq \begin{gathered}
\big(d^{{}^W}\!\!\zeta\big)(X_1,{\dots}, X_{k+1}) \= \sum_j({-}1)^{j+1}\Derc{W}_{X_j}\big(\zeta(X_1, {\dots}, \widehat{X_j} ,{\dots},X_{k+1})\big)+ \phantom{aaaaaaaaaaaaa}\\
\phantom{aaaaaaaaaaaaaaaaaaaaa}+\sum_{r<s}({-}1)^{r+s}\zeta([X_r,X_s],X_1,{\dots},\widehat{X_r},{\dots},\widehat{X_s},{\dots},X_{k+1}) \,\, .
\end{gathered} \eeq This exterior derivative splits into the sum $d^{{}^W}\!\! = \p^{{}^W}\!\!+\bar{\p}^{{}^W}\!\!$, with $$\p^{{}^W}\!\!: \W^{p,q}(M;W) \ra \W^{p+1,q}(M;W) \,\, , \quad \bar{\p}^{{}^W}\!\!: \W^{p,q}(M;W) \ra \W^{p,q+1}(M;W) \,\, .$$ Finally, consider the extension of the isometry $\vars: \L^{\bcdot}(T^{*\bC}M) \ra \eV M$ defined in Proposition \ref{propVM} to \beq \vars^{{}^W}\!\!: \L^{\bcdot}(T^{*\bC}M) {\otimes} W \ra \eV M {\otimes} W \,\, ,\quad \vars^{{}^W} \!\!\= \vars \,{\otimes} \Id_{{}_W} \,\, . \label{twistedro} \eeq To have a short notation for $\vars^{{}^W}\!\!$, for each decomposable element $\h {\otimes} s \in \L^{\bcdot}(T^{*\bC}M) {\otimes} W$ it is convenient to define $(\h {\otimes} s) \widehat{\bcdot} \xi^{\zero} \= (\h \widehat{\bcdot} \xi^{\zero}) {\otimes} s$, so that we can simply write $$\vars^{{}^W}\!\!(\zeta) = \frac1{2^{\frac{p+q}{2}}} \zeta \widehat{\bcdot} \xi^{\zero} \quad \text{ for every } \zeta \in \W^{p,q}(M;W)\,\, .$$

In analogy with Theorem \ref{isopartCD}, we have the following important identities for the $W$-twisted Chern-Dirac operators.

\begin{theorem} Let $\vars^{{}^W}\!\!: \L^{\bcdot}(T^{*\bC}M) {\otimes} W \ra \eV M {\otimes} W$ be the isometry \eqref{twistedro}. Then the partial Chern-Dirac operators \eqref{twpCD} verify \begin{align*}
\Big(\vars^{{}^W}\!\Big)^{{-}1} \!\circ \widetilde{\Diracde\!'{}^{{}^{(L)}}} \!\circ \vars^{{}^W}\! = \sqrt2 \p^{{}^W}\!\! \ &, \quad \Big(\vars^{{}^W}\!\Big)^{{-}1} \!\circ \widetilde{\Diracde\!''{}^{{}^{(L)}}} \!\circ \vars^{{}^W}\! = \sqrt2 \p^{{}^W\!*}\!\! \ , \\
\Big(\vars^{{}^W}\!\Big)^{{-}1} \!\circ \widetilde{\Diracde\!'{}^{{}^{(R)}}} \!\circ \vars^{{}^W}\! = \sqrt2 \bar{\p}^{{}^W\!*}\!\! \ &, \quad \Big(\vars^{{}^W}\!\Big)^{{-}1} \!\circ \widetilde{\Diracde\!''{}^{{}^{(R)}}} \!\circ \vars^{{}^W}\! = \sqrt2 \bar{\p}^{{}^W}\!\! \, . \end{align*} \label{isoparttwCD} \end{theorem}

\begin{proof} Consider a unitary frame field $(e_j): \cU \subset M \ra \U_{\!{}_{g,J}\!}(M)$, the associated normalized complex vectors $\e_j$, $\overline{\e}_j$ and a decomposable element $\zeta = \h {\otimes} s \in \W^{p,q}(M; W)$. Following the same arguments of the proof of Theorem \ref{isopartCD}, we have \begin{align*}
\widetilde{\Diracde\!'{}^{{}^{(L)}}}\!\big((\h \widehat{\bcdot} \xi^{\zero}) {\otimes} s\big) \!&= \sum_j \overline{\e}_j \!\bcdot_{\!\!{}_L} \widetilde{D}_{\e_j}\big((\h \widehat{\bcdot} \xi^{\zero}) {\otimes} s\big) - \frac12\sum_{r<s} (\overline{\e}_r \!\cdot\! \overline{\e}_s \!\cdot\! T_{rs}) \!\bcdot_{\!\!{}_L}\big((\h \widehat{\bcdot} \xi^{\zero}) {\otimes} s\big) \\
&=\sum_j \big(\overline{\e}_j \!\bcdot_{\!\!{}_L} (\Derc{\cC}_{\e_j}\h \widehat{\bcdot} \xi^{\zero})\big) {\otimes} s - \frac12\sum_{r<s} \big((\overline{\e}_r \!\cdot\! \overline{\e}_s \!\cdot\! T_{rs}) \!\bcdot_{\!\!{}_L}(\h \widehat{\bcdot} \xi^{\zero})\big) {\otimes} s +\\
&\phantom{aaaaaaaaaaaaaaaaaaaaaaaaaaaaaaaaaaii}+\sum_j \big(\overline{\e}_j \!\bcdot_{\!\!{}_L} (\h \widehat{\bcdot} \xi^{\zero})\big) {\otimes} \Derc{W}_{\e_j}s \\
&= (\p\h \widehat{\bcdot} \xi^{\zero}) {\otimes} s +\sum_j \big((\e^j \wedge \h)\widehat{\bcdot} \xi^{\zero}\big) {\otimes} \Derc{W}_{\e_j}s \phantom{aaaaaaaaaaaaaaaaaaaaaaa}\\
\phantom{\widetilde{\Diracde\!'{}^{{}^{(L)}}}\!\big((\h \widehat{\bcdot} \xi^{\zero}) {\otimes} s\big) \!}&= \Big(\p\h {\otimes} s + \sum_j (\e^j \wedge \h) {\otimes} \Derc{W}_{\e_j}s \Big) \widehat{\bcdot} \xi^{\zero} \\
\phantom{\widetilde{\Diracde\!'{}^{{}^{(L)}}}\!\big((\h \widehat{\bcdot} \xi^{\zero}) {\otimes} s\big) \!}&= \p^{{}^W}\!\!(\h {\otimes} s)\widehat{\bcdot} \xi^{\zero} \,\, .
\end{align*} Hence $$\Big(\Big(\vars^{{}^W}\!\Big)^{{-}1} \!\circ \widetilde{\Diracde\!'{}^{{}^{(L)}}} \!\circ \vars^{{}^W}\!\Big)(\h {\otimes} s) = \frac{2^{\frac{p+q+1}2}}{2^{\frac{p+q}2}}\p^{{}^W}\!\!(\h {\otimes} s) =\sqrt2 \p^{{}^W}\!\!(\h {\otimes} s) \ .$$ Proceeding in a similar way we obtain \begin{align*}
\widetilde{\Diracde\!''{}^{{}^{(L)}}}\!\big((\h \widehat{\bcdot} \xi^{\zero}) {\otimes} s\big) \!&= \sum_j \e_j \!\bcdot_{\!\!{}_L} \widetilde{D}_{\overline{\e}_j}\big((\h \widehat{\bcdot} \xi^{\zero}) {\otimes} s\big) - \frac12\sum_{r<s} (\e_r \!\cdot\! \e_s \!\cdot\! T_{\bar{r}\bar{s}}) \!\bcdot_{\!\!{}_L}\big((\h \widehat{\bcdot} \xi^{\zero}) {\otimes} s\big) \\
&=\sum_j \big(\e_j \!\bcdot_{\!\!{}_L} (\Derc{\cC}_{\overline{\e}_j}\h \widehat{\bcdot} \xi^{\zero})\big) {\otimes} s - \frac12\sum_{r<s} \big((\e_r \!\cdot\! \e_s \!\cdot\! T_{\bar{r}\bar{s}}) \!\bcdot_{\!\!{}_L}(\h \widehat{\bcdot} \xi^{\zero})\big) {\otimes} s +\\
&\phantom{aaaaaaaaaaaaaaaaaaaaaaaaaaaaaaaaaai}+\sum_j \big(\e_j \!\bcdot_{\!\!{}_L} (\h \widehat{\bcdot} \xi^{\zero})\big) {\otimes} \Derc{W}_{\overline{\e}_j}s \\
&= 2(\p^*\h \widehat{\bcdot} \xi^{\zero}) {\otimes} s +\sum_j \big({-}2(\e_j \lrcorner \h)\widehat{\bcdot} \xi^{\zero}\big) {\otimes} \Derc{W}_{\overline{\e}_j}s \phantom{aaaaaaaaaaaaaaaaaaaaaaaaaaaaaaaaaaaaaaaaaaaaaaaaaaaaaa}\\
&= 2\Big(\p^*\h {\otimes} s - \sum_j (\e_j \lrcorner \h) {\otimes} \Derc{W}_{\overline{\e}_j}s \Big) \widehat{\bcdot} \xi^{\zero} \\
&= 2\,\p^{{}^W\!*}\!(\h {\otimes} s)\widehat{\bcdot} \xi^{\zero}
\end{align*} and thus $$\Big(\Big(\vars^{{}^W}\!\Big)^{{-}1} \!\circ \widetilde{\Diracde\!''{}^{{}^{(L)}}} \!\circ \vars^{{}^W}\!\Big)(\h {\otimes} s) =2\frac{2^{\frac{p+q-1}2}}{2^{\frac{p+q}2}}\p^{{}^W\!*}\!(\h {\otimes} s) =\sqrt2 \p^{{}^W\!*}\!(\h {\otimes} s) \ .$$ The remaining two cases are perfectly analogous. \end{proof}

If we finally consider the {\it twisted left Chern-Dirac operator} and the {\it twisted right Chern-Dirac operator}, namely the operators \beq \widetilde{\Dirac{}^{\!\!\!{}^{(L)}}}\! = \widetilde{\Diracde\!'{}^{{}^{(L)}}} + \widetilde{\Diracde\!''{}^{{}^{(L)}}} \,\, , \quad\, \widetilde{\Dirac{}^{\!\!\!{}^{(R)}}}\! = \widetilde{\Diracde\!'{}^{{}^{(R)}}} + \widetilde{\Diracde\!''{}^{{}^{(R)}}} \,\, , \eeq by the same line of arguments on Theorem \ref{teoDdR}, we obtain

\begin{theo} Let $(M,g,J)$ be a compact Hermitian $2n$-manifold and $W$ be an Hermitian bundle over $M$ endowed with a fixed metric covariant derivative $\Derc{W}_{\color{white} X}\!\!$. Then $$\ker{\!\Big(\!\widetilde{\Dirac{}^{\!\!\!{}^{(L)}}}\!\! {+} \widetilde{\Dirac{}^{\!\!\!{}^{(R)}}}\!\Big)} \simeq \bigoplus_{k=0}^{2n} H^k_{d^{{}^W}}(M;W) \ , \quad \ker{\widetilde{\Dirac{}^{\!\!\!{}^{(R)}}}\!} \simeq \bigoplus_{p,q=0}^n H^{p,q}_{\bar{\p}^{{}^W}}(M;W) \,\, ,$$ where $H^k_{d^{{}^W}}(M;W)$ and $H^{p,q}_{\bar{\p}^{{}^W}}(M;W)$ are the $W$-valued De Rham and Dolbeault cohomology groups. \label{teoDdRtw} \end{theo}

\subsection{$\theta$-twisted $\eV$-spinor bundles}

The results of the previous section have immediate applications to the case of twisted cohomology groups $H^{\bcdot}_{d_{\theta}}(M;\bC)$ and $H^{\bcdot, \bcdot}_{\bar{\p}_{\theta}}(M)$ of Hermitian manifolds (see e.g. \cite{AD}). \smallskip

Let $(M,g,J)$ be an Hermitian $2n$-manifold with a fixed closed $1$-form $\theta$. Consider the trivial complex line bundle $\eL_{\theta}$ on $M$ endowed with the flat covariant derivative $\Derc{\theta}_{\color{white} X}\!\!$ defined for every global section $s$ of $\eL_{\theta}$ by $$\Derc{\theta}_{\color{white} X}\!\!s \= ds+\theta \otimes s \,\, .$$ Fixing an open covering $\{\cU_j\}$ of $M$ such that $\theta|_{\cU_j}=df_j$, we get an holomorphic trivialization $\{(\cU_j,e^{-f_j})\}$ of $\eL_{\theta}$, with transition functions $e^{f_j-f_k}$ on $\cU_j \cap \cU_k$, with respect to which $s^{\zero}=(\cU_j,e^{f_j})$ is a parallel nowhere vanish section. This gives rise to an Hermitian metric $h^{\theta}$ on $\eL_{\theta}$ with $h^{\theta}(s^{\zero},s^{\zero})=1$, so that $\Derc{\theta}_{\color{white} X}\!\!$ is metric w.r.t. $h^{\theta}$.

\begin{definition} The {\it $\theta$-twisted $\eV$-spinor bundle} is the tensor product bundle $$\eV_{\theta} M \= \eV M {\otimes} \eL_{-\theta} \,\, $$ endowed with the Hermitian metric $\check{h}^{\theta} \= \check{h} \otimes h^{-\theta}$ and the covariant derivative $\widetilde{\Derc{\theta}_{\color{white} i}}\! \= \Derc{\cC}_{\color{white}X}\!\! {\otimes} \Id_{\eL_{-\theta}} + \Id_{\eV M} {\otimes} \Derc{-\theta}_{\color{white}X}\!\!$ (see \S 5.1 for the definitions of $\check{h}$ and $\Derc{\cC}_{\color{white}X}\!\!$). \end{definition}

By Proposition \ref{Clifmulttw}, $\eV_{\theta} M$ is a Chern-Dirac bundle. The corresponding twisted partial Chern-Dirac operators $\Diracde\!'{}^{{}^{(L)}}_{\theta}$, $\Diracde\!''{}^{{}^{(L)}}_{\theta}$, $\Diracde\!'{}^{{}^{(R)}}_{\theta}$, $\Diracde\!''{}^{{}^{(R)}}_{\theta}$ are called {\it $\theta$-twisted partial Chern-Dirac operators}. Their sums $$\Dirac{}^{\!\!\!{}^{(L)}}_{\!\!\!\theta} = \Diracde\!'{}^{{}^{(L)}}_{\theta} + \Diracde\!''{}^{{}^{(L)}}_{\theta} \,\, , \quad \Dirac{}^{\!\!\!{}^{(R)}}_{\!\!\!\theta}\! = \Diracde\!'{}^{{}^{(L)}}_{\theta} + \Diracde\!''{}^{{}^{(L)}}_{\theta} \label{twthetaCD}$$ are called {\it $\theta$-twisted left Chern-Dirac operator} and {\it $\theta$-twisted right Chern-Dirac operator}. By means of the isomorphism between the De Rham complex of differential forms with values in $\eL_{-\theta}$ and the Lichnerowicz-Novikov complex $${\cdots} \overset{d_{\theta}}{\longrightarrow} \W^{k-1}(M;\bC) \overset{d_{\theta}}{\longrightarrow} \W^{k-1}(M;\bC) \overset{d_{\theta}}{\longrightarrow} {\cdots} \,\, ,\quad d_{\theta} \= d- \theta \wedge \,\, ,$$ from Theorem \ref{isoparttwCD} and Theorem \ref{teoDdRtw} we immediately get

\begin{theo} Let $(M,g,J)$ be an Hermitian $2n$-manifold and $$\vars^{{}^{\eL_{-\theta}}}\!\!: \L^{\bcdot}(T^{*\bC}M) {\otimes} \eL_{-\theta} \ra \eV_{\theta} M$$ the isometry defined in \eqref{twistedro} with $W= \eL_{-\theta}$. Then the $\theta$-twisted partial Chern-Dirac operators verify \begin{align*}
\Big(\vars^{{}^{\eL_{-\theta}}}\!\Big)^{{-}1} \!\circ \Diracde\!'{}^{{}^{(L)}}_{\theta} \!\circ \vars^{{}^{\eL_{-\theta}}}\! = \sqrt2 \p_{\theta} \,\, &, \quad \Big(\vars^{{}^{\eL_{-\theta}}}\!\Big)^{{-}1} \!\circ \Diracde\!''{}^{{}^{(L)}}_{\theta} \!\circ \vars^{{}^{\eL_{-\theta}}}\! = \sqrt2 \p^*_{\theta} \,\, , \\
\Big(\vars^{{}^{\eL_{-\theta}}}\!\Big)^{{-}1} \!\circ \Diracde\!'{}^{{}^{(R)}}_{\theta} \!\circ \vars^{{}^{\eL_{-\theta}}}\! = \sqrt2 \bar{\p}^*_{\theta} \,\, &, \quad \Big(\vars^{{}^{\eL_{-\theta}}}\!\Big)^{{-}1} \!\circ \Diracde\!''{}^{{}^{(R)}}_{\theta} \!\circ \vars^{{}^{\eL_{-\theta}}}\! = \sqrt2 \bar{\p}_{\theta} \,\, . \end{align*} In particular, if $M$ is compact, then $$\ker{\!\Big(\Dirac{}^{\!\!\!{}^{(L)}}_{\!\!\!\theta}\! {+} \Dirac{}^{\!\!\!{}^{(R)}}_{\!\!\!\theta}\!\Big)} \simeq \bigoplus_{k=0}^{2n} H^k_{d_{\theta}}(M;\bC) \ , \quad \ker{\Dirac{}^{\!\!\!{}^{(R)}}_{\!\!\!\theta}\!} \simeq \bigoplus_{p,q=0}^n H^{p,q}_{\bar{\p}_{\theta}}(M) \,\, ,$$ where $H^k_{d_{\theta}}(M;\bC)$ and $H^{p,q}_{\bar{\p}_{\theta}}(M)$ are the $\theta$-twisted De Rham and Dolbeault cohomology groups of $M$.\end{theo} 

\bigskip\bigskip
\font\smallsmc = cmcsc8
\font\smalltt = cmtt8
\font\smallit = cmti8
\hbox{\parindent=0pt\parskip=0pt
\vbox{\baselineskip 9.5 pt \hsize=5truein
\obeylines
{\smallsmc
Dipartimento di Matematica e Informatica ``Ulisse Dini'', Universit$\scalefont{0.55}{\text{\Aac}}$ di Firenze
Viale Morgagni 67/A, 50134 Firenze, ITALY}
\smallskip
{\smallit E-mail adress}\/: {\smalltt francesco.pediconi@unifi.it
}
}
}

\end{document}